\documentclass[reqno]{amsart}
\usepackage{blindtext}
\usepackage{fullpage}
\usepackage{mathtools}
\usepackage{longtable}
\usepackage{amsmath,amssymb,amsthm}
\usepackage{amscd}
\usepackage{bm}
\usepackage{hyperref}   
\usepackage{dsfont}
\usepackage{enumerate}
\usepackage{epsfig}
\usepackage{float, graphicx}
\usepackage{latexsym, amsxtra}
\usepackage{mathrsfs}
\usepackage{multicol}
\usepackage[normalem]{ulem}
\usepackage{psfrag}
\usepackage[parfill]{parskip}
\usepackage{stmaryrd}
\usepackage{tikz}
\usepackage[T1]{fontenc}
\usepackage{url}
\usepackage{verbatim}
\usepackage{indentfirst}
\usepackage{tikz-cd}
\usepackage{booktabs}
\usepackage{svg}

\usepackage{mathtools}

\usepackage{caption} 

\makeatletter
\def\thm@space@setup{%
	\thm@preskip=2ex \thm@postskip=2ex
}
\makeatother

\hypersetup{hidelinks}

\newtheorem{thm}{Theorem~}[section]
\newtheorem{lem}[thm]{Lemma~}

\newtheorem{prop}[thm]{Proposition~}
\newtheorem{ques}[thm]{Question~}
\newtheorem{cor}[thm]{Corollary~}

\theoremstyle{remark}
\newtheorem{rmk}[thm]{Remark~}

\theoremstyle{definition}  

\newcommand{\calF}{\mathcal{F}}
\newcommand{\calA}{\mathcal{A}}
\newcommand{\calG}{\mathcal{G}}
\newcommand{\calB}{\mathcal{B}}
\newcommand{\calI}{\mathcal{I}}

\newcommand{\CC}{\mathbb{C}}
\newcommand{\ZZ}{\mathbb{Z}}

\newcommand{\LL}{\mathbb{L}}
\newcommand{\PP}{\mathbb{P}}

\newcommand{\QQ}{\mathbb{Q}}

\newcommand{\HH}{\mathbb{H}}
\newcommand{\NN}{\mathbb{N}}

\newcommand{\V}{\mathcal{V}}

\newcommand\Span{\mathrm{Span}}

\newcommand\codim{\mathrm{codim}}

\DeclareMathOperator{\IF}{IF}
\DeclareMathOperator{\RF}{RF}
\DeclareMathOperator{\re}{Re}

\DeclareMathOperator{\Supp}{Supp}
\DeclareMathOperator{\Res}{Res}
\DeclareMathOperator{\mult}{mult}

\title{Combinatorial Nonresonance Theorems for Hyperplane Arrangement Complements}

\author[B. Xie]{Baiting Xie}
\address{Qiuzhen College, Tsinghua University, China}
\email{xbt23@mails.tsinghua.edu.cn}

\date{}

\begin{document}	

	\begin{abstract}
		We study the nonresonance phenomenon for complex rank‑one local systems on complements of hyperplane arrangements. We refine the method of Cohen, Dimca, and Orlik and obtain a combinatorial sufficient condition for nonresonance. As an application, we strengthen a theorem of Bailet, Dimca, and Yoshinaga by removing one of its conditions. We also develop restriction and lifting techniques to prove a nonresonance theorem for line arrangements.
	\end{abstract}
	
	\maketitle
    
	\section{Introduction}\label{section: introduction}

	 A classical problem in the theory of hyperplane arrangements is to determine which topological invariants of their complements are governed by their intersection combinatorics. For example, Orlik and Solomon \cite{MR558866} proved that the integral cohomology ring is combinatorially determined. However, Rybnikov's counter-example \cite{rybnikov2011fundamental} (see also \cite{MR2188450}) implies that the fundamental group is not. An important intermediate case between them is the cohomology with coefficients in a complex rank-one local system. This is also related to the question of whether the $ \CC $-cohomology groups of the associated Milnor fibers of the arrangements are combinatorially determined \cite{MR1310725}. See \cite{MR3661343} for a recent progress in this direction.
	 
	 In studying such cohomology groups, the phenomenon of nonresonance, namely the vanishing of cohomology outside the middle dimension, has received particular attention. A natural question is whether this property is determined by the combinatorial data of the arrangement together with that of the local system.

	 More explicitly, let $ \calA $ be a hyperplane arrangement in the complex projective space $ \PP^{n} $, with its complement denoted by $ M(\calA) $. Let $ \LL $ be a complex rank-one local system on $ M(\calA) $. The pair $ (\calA,\LL) $ is said to be \textbf{nonresonant} if  $ H^{p}(M(\calA),\LL) = 0 $ for all $ p \neq n $. The combinatorial data of $ \calA $ are encoded in its intersection lattice $ L(\calA) $. And a complex rank-one local system on $ M(\calA) $ is determined up to isomorphism by its counter-clockwise monodromy around the hyperplanes, yielding the following one‑to‑one correspondence:
	 \begin{equation*}
	 	\left\{
	 	\begin{tabular}{c}
	 		isomorphism classes of complex \\
	 		rank-one local systems $ \LL $ on $ M(\calA) $
	 	\end{tabular}
	 	\right\}
	 	\quad\stackrel{1:1}{\longleftrightarrow}\quad
	 	\left\{
	 	\begin{tabular}{c}
	 		$ m \colon \calA \rightarrow \CC^{\times} $ such that \\
	 		$ \prod\limits_{H \in \calA} m(H) = 1 $
	 	\end{tabular}
	 	\right\}.
	 \end{equation*}
	 Then the question raised in the previous paragraph can be stated concretely as:
	 \begin{ques}
	 	Given $ \calA $ and $ \LL $,
	 	do $ L(\calA) $ and $ m $ together determine whether the pair $ (\calA,\LL) $ is nonresonant?
	 \end{ques}
	 
	  A series of combinatorial sufficient conditions for nonresonance have been obtained by several authors, including Esnault-Schectman-Viehweg \cite{esnault1992cohomology}, Kohno \cite{MR846350}, and Schechtman-Terao-Varchenko \cite{schechtman1995local}. The most general known condition is due to Cohen, Dimca, and Orlik \cite{MR2038782} and is formulated in terms of resonant flats. Their precise definition, which depends only on $ L(\calA) $ and $ m $, is given in \S\ref{section: preliminary}. In the rank‑one case, their condition takes the following form:
	 \begin{thm}[{\cite[Theorem 1]{MR2038782}}]\label{thm: CDO condition for complex rank one}
	 	 Let $ \calA $ be a hyperplane arrangement in $ \PP^{n} $ and $ \LL $ a complex rank-one local system on the complement $ M(\calA) $.  Denote by $ \RF(\calA,\LL) $ the set of resonant flats of $ \calA $ with respect to $ \LL $. If there is a hyperplane $ H \in \calA $ such that $ H \notin \calF $ for any $ \calF \in \RF(\calA,\LL) $, then the pair $ (\calA,\LL) $ is nonresonant.
	 \end{thm}

	 When $ \calA $ is a line arrangement in $ \PP^{2} $, its intersection lattice is determined by the set of intersection points of the lines. Parallel to the notion of resonant flats, we call an intersection point $ p $ of $ \calA $ \textbf{resonant} (with respect to $ \LL $) if its multiplicity $ \mult(p) \geq 3 $ and $ \prod_{H \ni p}m(H) = 1 $. Then the condition in Theorem \ref{thm: CDO condition for complex rank one} simply says that there exists a line $ H \in \calA $ with $ m(H) \neq 1 $ containing no resonant points. Bailet, Dimca and Yoshinaga \cite{MR3858415} further studied the situation where there is a line containing exactly one resonant point and proved the following theorem.
	 
	 \begin{thm}[{\cite[Proposition 4.3]{MR3858415}}]\label{thm: BDY condition}
	 		 Let $ \calA $ be a line arrangement in $ \PP^{2} $ such that the lines in $ \calA $ do not all pass through a common point. Let $ \LL $ be a complex rank-one local system on the complement $ M(\calA) $ with monodromy map $ m \colon \calA \rightarrow \CC^{\times} $ such that $ m(H) \neq 1 $ for all lines $ H $ in $ \calA $. Suppose that
	 		 \begin{enumerate}
	 		 	\item there exists a line in $ \calA $  containing a unique resonant point, say $ p $, and
	 		 	\item no line through 
	 		 	$ p $ meets every line in $ \calA $ at a resonant point of $ \calA $.
	 		 \end{enumerate}
	 Then the pair $ (\calA,\LL) $ is nonresonant.
	 \end{thm}
	 It is natural to ask whether condition (2) can be dropped. This has been verified when $ \calA $ is a complexified real arrangements (see \cite{xie2025homology}, \cite{MR3090727} and \cite{MR3255765}).

	In this paper, we refine the method of Cohen, Dimca, and Orlik and obtain a strictly weaker combinatorial sufficient condition for nonresonance than the one given by Theorem \ref{thm: CDO condition for complex rank one}.
	\begin{thm}\label{main: n-vanishing}
				 Let $ \calA $ be a nonempty hyperplane arrangement in $ \PP^{n} $ and $ \LL $ a complex rank-one local system on the complement $ M(\calA) $.  Denote by $ \RF(\calA,\LL) $ the set of resonant flats of $ \calA $ with respect to $ \LL $. If for any non-identically-zero map $ \lambda \colon \RF(\calA,\LL) \rightarrow \NN $, the map 
				 \begin{equation*}
				 	\calA \longrightarrow \NN,\ H \mapsto 	\sum\limits_{\calF \ni H}\lambda(\calF)
				 \end{equation*}
				 is non-constant, then the pair $ (\calA,\LL) $ is nonresonant.
	\end{thm}

As an application of this criterion, we remove condition (2) from Theorem \ref{thm: BDY condition}. More precisely, we prove the following theorem.

\begin{thm}\label{main: one point}
	Let $ \calA $ be a line arrangement in $ \PP^{2} $ and $ \LL $ a complex rank-one local system on the complement $ M(\calA) $ with monodromy map $ m \colon \calA \rightarrow \CC^{\times} $. Suppose that
	\begin{enumerate}
		\item  there exists a line $ H_{1} \in \calA $ with $ m(H_{1}) \neq 1 $ that contains a unique resonant point $ p $, and
		\item there exists a line $ H_{2} \in \calA $ with $ m(H_{2}) \neq 1 $ that does not contain $ p $.
	\end{enumerate}
	 Then the pair $ (\calA,\LL) $ is nonresonant.
\end{thm}

Furthermore, we develop restriction and lifting techniques and establish the following nonresonance theorem for line arrangements, which is compatible with \cite[Theroem 2]{MR3056004}.

\begin{thm}\label{main: 2-vanishing}
	Let $ \calA $ be a line arrangement in $ \PP^{2} $ and $ \LL $ a complex rank-one local system on the complement $ M(\calA) $ with monodromy map $ m \colon \calA \rightarrow \CC^{\times} $. If there exists a bipartition $ \calA = \calA_{1} \sqcup \calA_{2} $ such that
	\begin{enumerate}
		\item for $ i=1,2 $, there exists $ H_{i} \in \calA_{i} $ with $ m(H_{i}) \neq 1 $, and
		\item for any lines $ H \in \calA_{1} $ and $ H' \in \calA_{2} $, their intersection point is not a resonant point,
	\end{enumerate}
	then the pair $ (\calA,\LL) $ is nonresonant.
\end{thm}

The paper is organized as follows. In \S\ref{section: preliminary}, we review the necessary background on the combinatorics and geometry of hyperplane arrangement complements. Our main nonresonance theorems are proved in \S\ref{section: vanishing}. 

\textbf{Acknowledgement.}
 The author thanks Yongqiang Liu, Shihao Wang and Chenglong Yu for their interest and helpful discussions.

\section{Preliminaries: Hyperplane Arrangements}\label{section: preliminary}

In this section, we recall some basic notion about combinatorics and geometry of hyperplane arrangement complements. Our goal is to introduce Theorem \ref{thm: algebraic de rham thm} and prove Lemma \ref{lem: long exact sequences}.

Throughout this section, let $ \calA $ be a hyperplane arrangement in $ \PP^{n}  $ and $ \LL $ a complex rank-one local system  on $ M(\calA) $ with monodromy map $ m \colon \calA \rightarrow \CC^{\times} $. 

\subsection{Combinatorics}\label{subsection: combinatorics}

For each subset $ S $ of $ \calA $, let $ Z(S) $ be the intersection of hyperplanes in $ S $. Denote by $ V(S) $ the cone of $ Z(S) $ and $  r(S) = \codim Z(S) $ the codimension of $ Z(S) $ in $ \PP^{n} $. In particular, note that $ V(S) = 0 $ and $ r(S) = n+1 $ when $ Z(S) = \emptyset $.

A subset $ \calF $ of $ \calA $ is called a \textbf{flat} if
\begin{equation*}
	\calF = \{H \in \calA \mid H \supseteq Z(\calF)\}.
\end{equation*} 
The \textbf{intersection lattice} $ L(\calA) $ is then defined to be the ranked lattice consisting of all flats of $ \calA $, ordered by inclusion. The \textbf{rank function} of $ L(\calA) $ is given by $ r \colon L(\calA) \rightarrow \NN $.  Furthermore, the monodromy map can be extended to $ m \colon L(\calA) \rightarrow \CC^{\times} $ as $ m(\calF) = \prod_{H \in \calF}m(H) $.

A nonempty flat $ \calF $ is called \textbf{irreducible} if there do not exist proper subsets $ S_{1}, S_{2} \subsetneq \calF $ such that $ r(S_{1}) + r(S_{2}) = r(\calF) $ and $ S_{1} \cup S_{2} = \calF $. Such a flat is called \textbf{resonant} (with respect to $ \LL $) if in addition $ r(\calF) \leq n $ and $ m(\calF) = 1 $. Denote by $ \IF(\calA) $ the set of irreducible flats of $ \calA $. Let $ \RF(\calA,\LL) \subseteq \IF(\calA) $ be the subset of resonant flats with respect to $ \LL $. The following lemma implies that irreducible flats are building blocks in $ L(\calA) $.

\begin{lem}\rm{(\cite{de1995hyperplane})}\label{lem: irreducible decomposition}
	Every nonempty flat $ \calF $ can be written as a disjoint union $ \calF = \calF_{1} \sqcup \calF_{2} \sqcup \cdots \sqcup \calF_{m} $ of irreducible flats such that $ r(\calF) = r(\calF_{1}) + r(\calF_{2}) + \cdots + r(\calF_{m}) $. Furthermore, such decomposition is unique and  satisfies the following properties:
	\begin{enumerate}
		\item Every irreducible flat $ \calG $ contained in $ \calF $ is contained in one of $ \calF_{i} $.
		\item Given any flat $ \calG_{i} \subseteq \calF_{i} $, the disjoint union $ \calG_{1} \sqcup \calG_{2} \sqcup \cdots \sqcup \calG_{m} $ is still a flat of $ \calA $.
	\end{enumerate}
\end{lem}
The disjoint union $ \calF = \calF_{1} \sqcup \calF_{2} \sqcup \cdots \sqcup \calF_{m} $ in Lemma \ref{lem: irreducible decomposition} is called the \textbf{irreducible decomposition} of $ \calF $, and the $ \calF_{i} $ are called \textbf{irreducible components} of $ \calF $.

Given a flat $ \calF $, one can associate two arrangements to $\calA$ at $\calF $, namely the localization and the restriction. They are defined as follows.

Note that the cone $ V(\calF) $ is a $ \CC $-linear subspace of $ \CC^{n+1} $, and for each $ H \in \calF $, its cone $ V(H) $ is a linear subspace of codimension $ 1 $ in $ \CC^{n+1} $ containing $ V(\calF) $. The \textbf{localization} of $ \calA $ at $ \calF $ is then defined to be the hyperplane arrangement
\begin{equation*}
	\calA_{\calF} = \{\PP(V(H)/V(\calF)) \mid H \in \calF\}
\end{equation*}
in $ \PP_{\calF} = \PP(\CC^{n+1}/V(\calF)) $. Each flat $ \calG \in L(\calA) $ contained in $ \calF $ corresponds to a flat
\begin{equation*}
	\calG_{\calF} = \{\PP(V(H)/V(\calF)) \mid H \in \calG\}
\end{equation*}
of $ \calA_{\calF} $. For any subset $ R \subseteq L(\calA) $, we set $ R_{\calF} = \{\calG_{\calF} \mid \calG \in R,\ \calG \subseteq \calF\} $. 

Note that the map $ \calG \mapsto \calG_{\calF} $ induces a bijection
\begin{equation*}
	\{\calG \in L(\calA) \mid \calG \subseteq \calF\} \xrightarrow{1:1} L(\calA_{\calF}). 
\end{equation*}
Under this identification, the map $ m \colon L(\calA) \rightarrow \CC^{\times} $ is restricted to a map $ L(\calA_{\calF}) \rightarrow \CC^{\times} $, which we still denote by $ m $. In particular, when $ m(\calF) = 1 $, we have $ \prod_{H' \in \calA_{\calF}} m(H) = 1 $. So $ m $ is the monodromy map of a complex rank-one local system on $ M(\calA_{\calF}) $, which we still denote by $ \LL $.

On the other hand, for each $ H \in \calA $ such that $ H \notin \calF $, $ H \cap Z(\calF) $ is a hyperplane in $ Z(\calF) $. The \textbf{restriction} of $ \calA $ at $ \calF $ is then defined to be the hyperplane arrangement
\begin{equation*}
	\calA^{\calF} = \{H\cap Z(\calF) \mid H \in \calA \setminus \calF \}
\end{equation*}
in $ Z(\calF) \simeq \PP^{n-r(\calF)} $. Each flat $ \calG \in L(\calA) $  corresponds to a flat
\begin{equation*}
 \calG^{\calF} = \{H' \in \calA^{\calF} \mid H' \supseteq Z(\calF)\cap Z(\calG)\}
\end{equation*}
of $ \calA^{\calF} $. For any subset $ R \subseteq L(\calA) $, we set $ R^{\calF} = \{\calG^{\calF} \mid \calG \in R \} $. 

Note that the map $ \calG \mapsto \calG^{\calF} $ induces a bijection
\begin{equation*}
	\{\calG \in L(\calA) \mid \calG \supseteq \calF\} \xrightarrow{1:1} L(\calA^{\calF}). 
\end{equation*}
Similar to the case of localizations, we restrict $ m $ to a map $ L(\calA^{\calF}) \rightarrow \CC^{\times} $. Moreover, when $ m(\calF) = 1 $, this induces a complex rank-one local system on $ M(\calA^{\calF}) $. We keep the same notation $ m $ and $ \LL $.

In order to study the behavior of $ \IF(\calA) $ under localization and restriction, we need to recall the notion of a building set. A building set $ \calB $ of $ \calA $ is a finite collection of nonempty flats such that $ \calB \supseteq \IF(\calA) $ and every nonempty flat $ \calF $ can be decomposed into a disjoint union $ \calF = \calF_{1} \sqcup \calF_{2} \sqcup \cdots \sqcup \calF_{m} $ where
\begin{enumerate}
	\item $ \calF_{1}, \cdots, \calF_{m} $ are the maximal elements of $ \{\calG \in \calB \mid \calG \subseteq \calF\} $, and
	\item $ r(\calF) = r(\calF_{1}) + r(\calF_{2}) + \cdots + r(\calF_{m}) $.
\end{enumerate}
 We further denote $ \calB^{0} = \{\calF \in \calB \mid r(\calF) \leq n\} $.

In particular, $ \IF(\calA) $ is a building set of $ \calA $. The following lemma follows directly from the definition.

\begin{lem}\rm{(\cite{de1995wonderful})}\label{lem: building sets under restrictions}
	Let $ \calF $ be a nonempty flat of $ \calA $. Then:
	\begin{enumerate} 
		\item  $ \IF(\calA)_{\calF} = \IF(\calA_{\calF}) $ and the map $ 	\calG  \mapsto  \calG_{\calF} $ induces a bijection
		\begin{equation*}
			\{\calG \in \IF(\calA) \mid \calG \subsetneq \calF\} \xrightarrow{1:1}   (\IF(\calA)_{\calF})^{0}.
		\end{equation*}
		\item If $ \calF \in IF(\calA)^{0} $, then $ \IF(\calA)^{\calF} $ is a building set of $ \calA^{\calF} $, and the map $ 	\calG  \mapsto  \calG^{\calF} $ induces a bijection
		\begin{equation*}
			\{\calG \in \IF(\calA) \mid \calG \supsetneq \calF \text{ and }r(\calG) \leq n \text{ or } \calG \cup \calF \text{ is a flat} \text{ and } r(\calG) + r(\calF) = r(\calG \cup \calF) \leq n \}  \xrightarrow{1:1}   (\IF(\calA)^{\calF})^{0}.
		\end{equation*}
	\end{enumerate}
\end{lem}
As a result, by abuse of notation, we also regard  $ (\IF(\calA)^{\calF})^{0} $ and $ (\IF(\calA)_{\calF})^{0} $ as subsets of $ \IF(\calA)^{0} $ via the inverses of bijections in Lemma \ref{lem: building sets under restrictions}. Furthermore, when $ \calF \in \RF(\calA,\LL) $, we denote $ \RF(\calA,\LL)^{\calF} = \RF(\calA,\LL) \cap (\IF(\calA)^{\calF})^{0} $ and $ \RF(\calA,\LL)_{\calF} = \RF(\calA,\LL) \cap (\IF(\calA)_{\calF})^{0} $.
\subsection{Wonderful Compactifications}\label{subsection: wonderful compactification}
Given a building set $ \calB $ of $ \calA $, one can construct a smooth projective variety $ Y(\calA;\calB) $ as follows. For each flat $ \calF \in \calB $, the linear projection $ \CC^{n+1} \rightarrow \CC^{n+1}/V(\calF) $ induces a regular morphism $ M(\calA) \rightarrow \PP_{\calF} = \PP(\CC^{n+1}/V(\calF)) $. Collecting all of them together, we obtain a regular morphism $ M(\calA) \rightarrow \prod_{\calF \in \calB} \PP_{\calF} $. The \textbf{wonderful compactification} $ Y(\calA;\calB) $ of $ M(\calA) $ with respect to $ \calB $ is then defined to be the closure of the graph $ \Gamma(\calA;\calB) $ of this morphism in $ \PP^{n} \times \prod_{\calF \in \calB} \PP_{\calF} $.

  \begin{lem}\label{lem: wonderful compactification}\rm{(\cite{de1995wonderful})}
  	$ Y(\calA;\calB) $ is a smooth complex variety satisfying the following condition:
  	\begin{enumerate}
  		\item The graph $ \Gamma(\calA;\calB) $ is a Zariski-open subset of $ Y(\calA;\calB) $, with its complement $ D(\calA;\calB) = Y(\calA;\calB) \setminus \Gamma(\calA;\calB) $ being a normal crossing divisor. The irreducible components $ E(\calF;\calB) $ of $ D(\calA;\calB) $ are smooth and indexed by the flats $ \calF \in \calB^{0} $. 
  		\item The projection $ \PP^{n} \times \prod_{\calF \in \calB} \PP_{\calF} \rightarrow \PP^{n} $ induces a birational morphism $ \pi \colon Y(\calA) \rightarrow \PP^{n} $. The restriction $ \pi|_{\Gamma(\calA;\calB)} \colon \Gamma(\calA;\calB) \rightarrow M(\calA) $ is an isomorphism.
  	\end{enumerate}
 \end{lem}

As a result, we can identify $ \Gamma(\calA;\calB) $ with $ M(\calA) $ and regard $ Y(\calA) $ as a compactification of $ M(\calA) $.

In this paper, we mainly focus on the case where $ \calB = \IF(\calA) $. In this case, we drop the building set from all notation for simplicity. The following lemma implies that $ \pi \colon Y(\calA) \rightarrow \PP^{n} $ is exactly the blow-up constructed in \cite[Theorem 3.1]{schechtman1995local}.

\begin{lem}\rm{(\cite{de1995hyperplane})}
	The birational morphism $ \pi \colon Y(\calA) \rightarrow \PP^{n} $ factors as a composition
	\begin{equation*}
		\begin{tikzcd}
			Y(\calA) \simeq Y_{0} \arrow{r}{\pi_{1}} & Y_{1}  \arrow{r}{\pi_{2}} & \cdots  \arrow{r}{\pi_{n-1}} & Y_{n-1}  \arrow{r}{\pi_{n}} & Y_{n} = \PP^{n}
		\end{tikzcd}
	\end{equation*}
	with $ \pi_{k} \colon Y_{k-1} \rightarrow Y_{k} $ the blow-up along the proper transform of $ Z_{k} $ under $ \pi_{k+1} \circ \cdots \circ \pi_{n} $, where $ Z_{k} $ is the union of $ Z(\calF) $ for all irreducible flats $ \calF $ of rank $ k $. Furthermore, for any irreducible flat $ \calF $ of rank $ k $, the divisor $ E(\calF) $ can be obtained as follows: take the proper transform of $ Z(\calF) $ under $ \pi_{k+1} \circ \cdots \circ \pi_{n} $; then take the exceptional divisor of $ \pi_{k} $ corresponding to this subvariety; finally, $ E(\calF) $ is the proper transform of this exceptional divisor under $ \pi_{1} \circ \cdots \circ \pi_{k-1} $.
\end{lem}

 The following lemma describes $ E(\calF) $.

 \begin{lem}\label{lem: EF}\rm{(\cite{de1995wonderful})}
 	Given an irreducible flat $ \calF $, there exists a canonical isomorphism
 	\begin{equation*}
 		E(\calF) \simeq Y(\calA^{\calF};\IF(\calA)^{\calF}) \times Y(\calA_{\calF}).
 	\end{equation*}
 	 Furthermore, under this isomorphism, for any $ \calG \in \IF(\calA)^{0} $ such that $ \calG \neq \calF $, we have
 	\begin{equation*}
 		E(\calG) \cap  E(\calF) = \begin{cases}
 			E(\calG^{\calF};\IF(\calA)^{\calF}) \times Y(\calA_{\calF}), & \calG \in (\IF(\calA)^{\calF})^{0}, \\
 			Y(\calA^{\calF};\IF(\calA)^{\calF}) \times E(\calG_{\calF}), & \calG \in (\IF(\calA)_{\calF})^{0}, \\
 			\emptyset, & \text{otherwise.} \\
 		\end{cases}
 	\end{equation*}
 	 In particular, letting
 	\begin{equation*}
 		E(\calF)^{\circ} = E(\calF) \setminus \Supp(D(\calA)-E(\calF)),
 	\end{equation*}
 	we have $ E(\calF)^{\circ} = M(\calA^{\calF}) \times M(\calA_{\calF}) $ under the above isomorphism.
 \end{lem}

Recall that $ \LL $ is a complex rank-one local system on $ M(\calA) $ with monodromy map $ m $. By Lemma \ref{lem: wonderful compactification}, the monodromy of $ \LL $ around $ E(\calF) $ is exactly $ m(\calF) = \prod_{H \in \calF}m(H) $. In particular, the monodromy around $ E(\calF) $ is trivial if and only if $ E(\calF)  \in \RF(\calA,\LL) $. So the local system $ \LL $ can be extended to one on $ E(\calF)^{\circ} $ for any $ \calF \in \RF(\calA,\LL) $. Moreover, for any irreducible flat $ \calG \neq \calF $ such that $ 	E(\calG) \cap  E(\calF) \neq \emptyset $, the monodromy around $ 	E(\calG) \cap  E(\calF) $ is exactly $ m(\calG) $, the monodromy around $ E(\calG) $. So the restriction of this extension to $ E(\calF)^{\circ} $ is compatible with the local systems on $ M(\calA^{\calF}) $ and $ M(\calA_{\calF}) $ defined in Subsection \ref{subsection: combinatorics}.

Furthermore, the local system $ \LL $ can be extended to one on the open set
\begin{equation*}
	Y(\calA) \setminus \Supp\left(D(\calA) - \sum\limits_{\calF \in \RF(\calA,\LL)}E(\calF)\right) \supseteq M(\calA).
\end{equation*}
All local systems in this section will be the restrictions of this one. By abuse of notation, we always denote them by $ \LL $.

\subsection{Extensions of Local Systems}\label{subsection: Extension}

Recall that $ \LL $ is a complex rank-one local system on $ M(\calA) = Y(\calA)\setminus D(\calA) $.  By Deligne's Riemann-Hilbert correspondence, there exists a regular meromorphic connection $ (\mathcal{M},\nabla) $ extending $ \LL $ on $ Y(\calA) $, which is unique up to isomorphism. This can be explicitly constructed as follows.

 For each flat $ \calF $, let $ \alpha(\calF) $ be the complex number such that  $ \exp(-2\pi \sqrt{-1} \alpha(\calF)) = m(\calF) $ and the real part of $ \alpha(\calF) $ lies in $ [0,1) $. Denote by $ d(\calF) = \alpha(\calF) - \sum_{H \in \calF}\alpha(H) $. By definition, $ d(\calF) $ is always an integer.

Denote by $ O $ the sheaf of holomorphic functions on $ Y(\calA) $ and $ O[D(\calA)] $ the sheaf of meromorphic functions on $ Y(\calA) $ that are holomorphic on $ M(\calA) $. Choose a defining polynomial $ l_{H} $ for each hyperplane $ H \in \calA $. Recall that we have defined a birational morphism $ \pi \colon Y(\calA) \rightarrow \PP^{n} $. For any generic hyperplane $ H_{0} $ in $ \PP^{n} $ and its defining polynomial $ l_{H_{0}} $, $  Q_{H,H_{0}} = \frac{l_{H}}{l_{H_{0}}} \circ \pi $ is a holomorphic function on the open set $ U_{H_{0}} = Y(\calA) \setminus \pi^{-1}H_{0} $ and we can define a meromorphic connection along $ D(\calA)|_{U_{H_{0}}} $:
\begin{equation}\label{eqn: connection on Y}
	\nabla \colon O_{U_{H_{0}}}[D(\calA)|_{U_{H_{0}}}] \longrightarrow \Omega^{1}_{U_{H_{0}}}\otimes O_{U_{H_{0}}}[D(\calA)|_{U_{H_{0}}}],\ g  \mapsto dg +  g \cdot \sum\limits_{H \in \calA}\alpha(H)\frac{dQ_{H,H_{0}}}{Q_{H,H_{0}}}.
\end{equation} 

Denote by $ \mathcal{M} = \pi^{*}O_{\PP^{n}}(d(\calA)) \otimes O[D(\calA)] $. Note that $ d(\calA) = - \sum_{H \in \calA} \alpha(H) $ since $ \alpha(\calA) = 0 $. So by gluing all these $ \nabla $ together, we obtain a meromorphic connection on $ Y(\calA) $ along $ E $:
\begin{equation*}
	\nabla \colon \mathcal{M} \longrightarrow \Omega_{Y(\calA)}^{1} \otimes \mathcal{M}.
\end{equation*}

By Formula (\ref{eqn: connection on Y}), the meromorphic connection $ (\mathcal{M},\nabla) $ is regular. Furthermore, the monodromy of its horizontal section around each hyperplane $ H \in \calA $ is exactly $ \exp(-2\pi \sqrt{-1} \alpha(H)) = m(H) $. Therefore, the meromorphic connection $ (\mathcal{M},\nabla) $ is the regular meromorphic extension of $ \LL $ on $ Y(\calA) $.

The following lemma explicitly constructs the canonical extension of $ \LL $ on $ Y(\calA) $ with logarithmic poles along $ D(\calA) $.

\begin{lem}\label{lem: canonical extension}
 Define
	\begin{equation*}
		\V(\calA) = \pi^{*}O_{\PP^{n}}(d(\calA)) \otimes O(-\sum\limits_{\calF \in \IF(\calA)^{0}}d(\calF)E(\calF)) \subset \mathcal{M}.
	\end{equation*}
	Then $ (\mathcal{M},\nabla) $ has logarithmic poles along $ D(\calA) $ with respect to $ \V(\calA) $, and the residue along $ E(\calF) $ with respect to $ \V(\calA) $ is
	\begin{equation*}
		\Res_{E(\calF)}(\V(\calA)) = \alpha(\calF).
	\end{equation*}
In particular, the real part of $ \Res_{E(\calF)}(\V(\calA)) $ lies in $ [0,1) $ for any $ \calF \in \IF(\calA)^{0} $.
\end{lem}

\begin{proof}
	For any point $ y $ in $ Y(\calA) $, let $ U_{y} $ be a small neighbourhood of $ y $ and choose a local coordinate $ \{z_{1},\cdots,z_{n}\} $ such that $ D(\calA) $ is locally given by $ z_{1}\cdots z_{m} = 0 $ in $ U_{y} $. Let $ E(\calF_{i}) $ be the irreducible component of $ E $ such that $ E(\calF_{i}) \cap U_{y} = \{z_{i}=0\} $. The identification $ \mathcal{M}|_{U_{y}} \simeq O_{U_{y}}[D(\calA)|_{U_{y}}] $ yields
	\begin{equation*}
		\V(\calA)|_{U_{y}} \simeq O_{U_{y}}(-\sum_{i=1}^{m}d(\calF_{i})E(\calF_{i})) = O_{U_{y}} \cdot \prod\limits_{i=1}^{m}z_{i}^{d(\calF_{i})}.
	\end{equation*}
	
	Fix a generic hyperplane $ H_{0} $ such that $ U_{y} \cap \pi^{-1}H_{0} = \emptyset $ and define $ l_{H_{0}} $, $ Q_{H,H_{0}} $ as above. Then by Formula (\ref{eqn: connection on Y}), we have
	\begin{equation*}
		\nabla|_{U_{y}} \colon  O_{U_{y}}[D(\calA)|_{U_{y}}] \longrightarrow \Omega^{1}_{U_{y}}\otimes O_{U_{y}}[D(\calA)|_{U_{y}}],\ g  \mapsto dg +  g \cdot \sum\limits_{H \in \calA}\alpha(H)\frac{dQ_{H,H_{0}}}{Q_{H,H_{0}}}.
	\end{equation*}

 The restriction of $ Q_{H,H_{0}} $ on $ U_{y} $ has the form
 \begin{equation*}
 	Q_{H,H_{0}}|_{U_{y}} = g_{H}\prod\limits_{\calF_{i} \ni H}z_{i},
 \end{equation*}
 where $ g_{H} $ is a holomorphic function on $ U_{y} $ such that $ g_{H}(y) \neq 0 $. 
 
 So for any holomorphic function $ g $ on some open subset of $ U_{y} $, we have
 \begin{eqnarray*}
 	\nabla(g\cdot \prod\limits_{i=1}^{m}z_{i}^{d(\calF_{i})}) & = & \left(dg + g\cdot \sum\limits_{i=1}^{m}d(\calF_{i})\frac{dz_{i}}{z_{i}}+  g \cdot \sum\limits_{H \in \calA}\alpha(H)\frac{dQ_{H,H_{0}}}{Q_{H,H_{0}}}\right)\cdot \prod\limits_{i=1}^{m}z_{i}^{d(\calF_{i})} \\
 	& = & \left(dg + g\cdot \left(\sum\limits_{i=1}^{m}\left(d(\calF_{i})+\sum\limits_{H \in \calF_{i}}\alpha(H)\right)\frac{dz_{i}}{z_{i}}+  \sum\limits_{H \in \calA}\alpha(H)\frac{dg_{H}}{g_{H}}\right)\right)\cdot \prod\limits_{i=1}^{m}z_{i}^{d(\calF_{i})} \\
 		& = & \left(dg + g\cdot \left(\sum\limits_{i=1}^{m}\alpha(\calF_{i})\frac{dz_{i}}{z_{i}}+  \sum\limits_{H \in \calA}\alpha(H)\frac{dg_{H}}{g_{H}}\right)\right)\cdot \prod\limits_{i=1}^{m}z_{i}^{d(\calF_{i})}.
 \end{eqnarray*}
 Therefore, the image of $ \V(\calA) $ under $ \nabla $ is contained in $ \Omega_{Y(\calA)}^{1}(\log D(\calA)) \otimes \V(\calA) $. So the meromorphic connection $ (\mathcal{M},\nabla) $ has logarithmic poles along $ D(\calA) $ with respect to $ \V(\calA) $. Furthermore, for any $ \calF \in \IF(\calA)^{0} $, take $ y $ as a generic point on $ E(\calF) $. Then the equation above implies that $ \Res_{E(\calF)}(\V(\calA)) = \alpha(\calF) $.
\end{proof}

\begin{cor}\label{cor: classifying extension with log poles}
	Let $ (\V_{0},\nabla_{0}) $ be an extension of $ \LL $ on $ Y(\calA) $ with logarithmic poles along $ D(\calA) $. Then there exists a unique map $ \mu \colon \IF(\calA)^{0} \rightarrow \ZZ $ such that $ \V_{0} \simeq \V(\calA) \otimes O(-\sum_{\calF \in \IF(\calA)^{0}}\mu(\calF)E(\calF)) $ as connections.
\end{cor}

\begin{proof}
	By the uniqueness of $ \mathcal{M} $, we may assume $ \V_{0} \subseteq \mathcal{M} $ and $ \nabla_{0} = \nabla $. Take
	\begin{equation*}
		\mu(\calF) = \Res_{E(\calF)}(\V_{0}) - \alpha(\calF).
	\end{equation*}
	Then $ \mu(\calF)  \in \ZZ $ since $ \exp(-2\pi\sqrt{-1}\alpha(\calF)) = m(\calF) = \exp(-2\pi\sqrt{-1}\Res_{E(\calF)}(\V_{0})) $. Furthermore, take
		\begin{equation*}
		\V_{0}' = \V_{0} \otimes O(\sum\limits_{\calF \in \IF(\calA)^{0}}\mu(\calF)E(\calF)) \subset M.
	\end{equation*}
	Similar to the proof in Lemma \ref{lem: canonical extension}, we have $ (\mathcal{M},\nabla) $ has logarithmic poles along $ E $ with respect to $ \V_{0}' $, and the residue along $ E(\calF) $ with respect to $ \V_{0}' $ is $ \Res_{E(\calF)}(\V_{0}) - \mu(\calF) = \alpha(\calF) $. So by \cite[Theorem 5.2.17]{hotta2007d}, we have $ \V_{0}' \simeq \V(\calA) $ as connections, which implies existence. Uniqueness is obvious.
\end{proof}

\begin{rmk}\label{rmk: residue classify extension}
	For the same reason, given a compact smooth complex manifold $ Y $ and a simple normal crossing divisor $ D $, the extensions of a given complex rank-one local system on $ Y \setminus D $ with logarithmic poles along $ D $ can be classified by their residues uniquely.
\end{rmk}

\subsection{Logarithmic de Rham Complexes}\label{subsection: log de Rham complexes}
We first recall Deligne's algebraic de Rham theorem.
\begin{thm}\label{thm: algebraic de rham thm}\rm{(\cite[Corollary 6.10]{deligne1970equations})}
				Let $ Y $ be a smooth complex variety, $ D $ a normal crossing divisor on $ Y $, $ \LL $ a complex local system  on $  Y \setminus D $, $ (\V,\nabla) $ an extension of $ \LL $ on $ Y $ with logarithmic poles along $ D $. If for each irreducible component of $ D $, no eigenvalue of the residue is a positive integer, then for any $ p \geq 0 $,
			\begin{equation*}
					\HH^{p}(Y,\Omega^{*}_{Y}(\log D)\otimes \V) \simeq H^{p}(Y \setminus D ,\LL).
				\end{equation*}
\end{thm}

Returning to the setting of hyperplane arrangements, we focus on extensions of $ \LL $ of the form
\begin{equation*}
	\V(\calA,R) = \V(\calA) \otimes O(-\sum\limits_{\calF \in R}E(\calF)),
\end{equation*}
where $ R $ is a subset of $ \IF(\calA)^{0} $. Similar to the proof in Lemma \ref{lem: canonical extension}, the residue along $ E(\calF) $ with respect to $ \V(\calA,R) $ is exactly
\begin{equation}\label{eqn: residue of V}
	\Res_{E(\calF)}\V(\calA,R) = \begin{cases}
		\alpha(\calF)+1, & \calF \in R, \\
		\alpha(\calF), & \calF \notin R. 
	\end{cases}
\end{equation}

In particular, for any $ T \subseteq \RF(\calA,\LL) \setminus R $, the pair $ (\V(\calA,R),\nabla) $ is a connection with logarithmic poles along $ D(\calA) - \sum_{\calF \in T}E(\calF) $. We denote the hypercohomology of its logarithmic de Rham complex by 

\begin{equation*}
	H^{p}(\calA,R,T) = \HH^{p}\left(Y(\calA),\Omega^{*}_{Y(\calA)}\left(\log\left( D(\calA) - \sum\limits_{\calF \in T}E(\calF)\right)\right)\otimes \V(\calA,R)\right).
\end{equation*}
We omit $ T $ if $ T = \emptyset $. In that case, we also omit $ R $ if $ R = \emptyset $.

We further study the relationships among these hypercohomology groups. 
\begin{lem}\label{lem: long exact sequences}
		For any subset $ R \subset \IF(\calA)^{0} $ and any resonant flat $ \calF $ such that $ \calF \notin R $ and $ 	R \cap \RF(\calA,\LL)^{\calF} = \emptyset $, there exist two long exact sequences of $ \CC $-vector spaces:
	\begin{equation}\label{eqn: long exact sequence 1}
		\begin{split}
					& \cdots \longrightarrow H^{p}(\calA,R,\{\calF\}) \longrightarrow H^{p}(\calA,R) \longrightarrow \bigoplus\limits_{q+r=p-1}H^{q}(\calA^{\calF}) \otimes H^{r}(\calA_{\calF},R_{\calF}) \\ 
					& \longrightarrow H^{p+1}(\calA,R,\{\calF\}) \longrightarrow H^{p+1}(\calA,R) \longrightarrow \cdots
		\end{split}
	\end{equation}
	and
	\begin{equation}\label{eqn: long exact sequence 2}
			\begin{split}
			& \cdots \longrightarrow H^{p}(\calA,R \cup \{\calF\}) \longrightarrow H^{p}(\calA,R,\{\calF\}) \longrightarrow \bigoplus\limits_{q+r=p}H^{q}(\calA^{\calF}) \otimes H^{r}(\calA_{\calF},R_{\calF}) \\ 
			& \longrightarrow H^{p+1}(\calA,R \cup \{\calF\}) \longrightarrow H^{p+1}(\calA,R,\{\calF\}) \longrightarrow \cdots
			\end{split}
	\end{equation}
\end{lem}

\begin{proof}
	
	For simplicity, let $ Y = Y(\calA) $, $ D = D(\calA) $, $ E_{0} = E(\calF) $, and $ \V = \V(\calA,R) $.
	
	As shown in \cite[Properties 2.3]{esnault1992lectures}, for any $ p \in \ZZ $, there are two short exact sequences of sheaves:
	\begin{equation*}
		0 \longrightarrow \Omega^{p}_{Y}(\log(D - E_{0}))\otimes \V \longrightarrow \Omega^{p}_{Y}(\log D)\otimes \V \longrightarrow \Omega^{p-1}_{E_{0}}(\log(D - E_{0})|_{E_{0}})\otimes \V|_{E_{0}} \longrightarrow 0
	\end{equation*}
	and
	\begin{equation*}
		0 \longrightarrow \Omega^{p}_{Y}(\log D)\otimes \V \otimes O_{Y}(-E_{0}) \longrightarrow  \Omega^{p}_{Y}(\log(D - E_{0}))\otimes \V \longrightarrow \Omega^{p}_{E_{0}}(\log(D - E_{0})|_{E_{0}})\otimes \V|_{E_{0}} \longrightarrow 0.
	\end{equation*}

	Since $ \calF \in \RF(\calA,\LL) \setminus R $, the residue of $  \V $ along $  E_{0} $ is $ 0 $. So the restriction of $ (\V,\nabla) $ on $ E_{0} $ defines an extension of $ \LL|_{E(\calF)^{\circ}} $ with logarithmic poles, whose logarithmic de Rham complex is exactly $ \Omega^{*}_{E_{0}}(\log(D - E_{0})|_{E_{0}})\otimes \V|_{E_{0}} $. Then it follows directly from the definition that the maps in above short exact sequences induce morphisms between complexes. So by taking the long exact sequences of hypercohomology, it suffices to prove that for any $ p \geq 0 $,
	\begin{equation}\label{eqn: iso on hypercohomology given by restriction}
		\HH^{p}(E_{0},\Omega^{*}_{E_{0}}(\log(D - E_{0})|_{E_{0}})\otimes \V|_{E_{0}}) \simeq \bigoplus\limits_{q+r=p}H^{q}(\calA^{\calF}) \otimes H^{r}(\calA_{\calF},R_{\calF}).
	\end{equation}

	 By Lemma \ref{lem: EF}, we have $ E_{0} = E(\calF) \simeq Y(\calA^{\calF};\IF(\calA)^{\calF}) \times Y(\calA_{\calF}) $. For simplicity, let $ Y_{1} = Y(\calA^{\calF};\IF(\calA)^{\calF}) $, $ Y_{2} = Y(\calA_{\calF})$, $ D_{1} = D(\calA^{\calF};\IF(\calA)^{\calF}) $, and $ D_{2} = D(\calA_{\calF}) $. Let $ \V_{i}(i=1,2) $ be the restriction of $ \V $ on $ Y_{i} \subseteq E_{0} $. Then by remark \ref{rmk: residue classify extension}, we have $ \V|_{E_{0}} \simeq \V_{1} \boxtimes \V_{2} $, which induces an isomorphism between complexes
	 \begin{equation}\label{eqn: iso on complexes given by restriction}
	 	\Omega^{*}_{E_{0}}(\log(D - E_{0})|_{E_{0}})\otimes \V|_{E_{0}} \simeq 	\left(\Omega^{*}_{Y_{1}}(\log D_{1})\otimes \V_{1}\right) \boxtimes \left(\Omega^{*}_{Y_{2}}(\log D_{2})\otimes \V_{2}\right).
	 \end{equation}
	 
	  Furthermore, by Formula (\ref{eqn: residue of V}) we have for any $ \calG \in (\IF(\calA)^{\calF})^{0} $,
	\begin{equation*}
		\Res_{E(\calG^{\calF})}\V_{1} = \Res_{E(\calG) \cap  E_{0}}\V|_{E_{0}}  = 	\Res_{E(\calG)}\V = \begin{cases}
			\alpha(\calG), & \calG \notin R, \\
				\alpha(\calG)+1, & \calG \in R.
		\end{cases}
	\end{equation*}
	Since $ R \cap \RF(\calA,\LL)^{\calF} = \emptyset $ and $ \re \alpha(\calG) \in [0,1) $, no residue of $ \V_{1} $ is a positive integer. So by Theorem \ref{thm: algebraic de rham thm} we have $\HH^{q}(Y_{1},\Omega^{*}_{Y_{1}}(\log E_{1})\otimes \V_{1}) \simeq H^{q}(M(\calA^{\calF}),\LL) \simeq H^{q}(\calA^{\calF}) $. On the other hand, by Formula (\ref{eqn: residue of V}),  we have for any $ \calG \in  (\IF(\calA)_{\calF})^{0} $,
	\begin{equation*}
		\Res_{E(\calG_{\calF})}\V_{2} = \Res_{E(\calG) \cap  E_{0}}\V|_{E_{0}}  = \Res_{E(\calG_{\calF})}\V(\calA_{\calF},R_{\calF}). 
	\end{equation*}
 So by Remark \ref{rmk: residue classify extension} we have $ \V_{2} \simeq \V(\calA_{\calF};R_{\calF}) $ as connections. So Formula (\ref{eqn: iso on hypercohomology given by restriction}) follows directly from the isomorphism (\ref{eqn: iso on complexes given by restriction}) by the generalized K\"unneth formula.
\end{proof}

\section{Nonresonance Theorems}\label{section: vanishing}

In this section, we prove our main theorems. The proof relies on two lemmas. Lemma \ref{lem: vanishing lemma} yields a direct proof of Theorem \ref{main: n-vanishing}, which in turn implies Theorem \ref{main: one point}. Lemma \ref{lem: dimension shift} has two applications: a restriction argument yields Proposition \ref{prop: k-vanishing given by higher ranked resonant flats}, and a lifting technique allows us to prove Theorem \ref{main: 2-vanishing}.

 Throughout this section, let $ \calA $ be a hyperplane arrangement in $ \PP^{n} $ and $ \LL $ a complex rank-one local system on $ M(\calA) $ with monodromy map $ m \colon \calA \rightarrow \CC^{\times} $. We keep all notation from Section \ref{section: preliminary}. 
 
 For any subset $ R \subseteq \IF(\calA)^{0} $, we call the pair $ (\calA,R) $ \textbf{$ k $-vanishing} if $ H^{p}(\calA,R) = 0 $ for any $ p < k $. In particular, when $ R = \emptyset $, we omit $ R $. By Theorem \ref{thm: algebraic de rham thm},  we have $ H^{p}(M(\calA),\LL) = 0 $ for any $ p < k $ if and only if there exists some subset $ R \subseteq \IF(\calA)^{0} $ such that $ R \cap \RF(\calA,\LL) = \emptyset $ and $ (\calA,R) $ is $ k $-vanishing. Furthermore, $ H^{p}(\calA,\LL) = 0 $ holds for all $ p > n $ by Artin's vanishing theorem. Thus, the pair $ (\calA,\LL) $ is nonresonant if and only if there exists some subset $ R \subseteq \IF(\calA)^{0} $ such that $ R \cap \RF(\calA,\LL) = \emptyset $ and $ (\calA,R) $ is $ n $-vanishing.

\subsection{Main Tools}\label{subsection: main tools}

The following vanishing lemma is a corollary of \cite[Theorem 6.2]{esnault1992lectures}, which is implicit in the proof of \cite[Theorem 1.5]{xie2025n}. Here we include a proof for readers' convenience.

\begin{lem}\label{lem: vanishing lemma}
	Suppose that $ \calA \neq \emptyset $. Let $ R $ be a subset of $  \IF(\calA)^{0} $. If there exists a map $ \delta \colon \calA \rightarrow \QQ $ such that
	\begin{enumerate}
		\item the sum $ \sum\limits_{H \in \calA}\delta(H) = 0 $,
		\item for any $ \calF \in \IF(\calA)^{0} $ such that $ \re \alpha(\calF) = 0 $, the sum $ \sum\limits_{H \in \calF}\delta(H) \neq 0 $ , and
		\item $ R = \{\calF \in \IF(\calA)^{0} \mid  \sum_{H \in \calF}\delta(H)< 0 \text{ and } \re \alpha(\calF) = 0 \}$,
	\end{enumerate}
	then $ (\calA,R) $ is $ n $-vanishing.
\end{lem}

\begin{proof}
	 For any $ \calF \in \IF(\calA)^{0} $, denote by
	\begin{equation*}
		\epsilon(\calF) = \begin{cases}
			 1 + \re \alpha(\calF) + \sum\limits_{H \in \calF}\delta(H)  & \calF \in R \\
			\re \alpha(\calF) + \sum\limits_{H \in \calF}\delta(H)  & \text{otherwise}
		\end{cases}
	\end{equation*}
	
	By rescaling $ \delta $, we may assume that $ \epsilon(\calF) \in (0,1) $ for any $ \calF \in \IF(\calA)^{0} $. Choose $  N \in \NN $ such that $  N\epsilon(\calF)(\calF \in \IF(\calA)^{0}) $ are all integers. Note that $ d(\calA) = -\sum_{H \in \calA} \alpha(H) = -\sum_{H \in \calA} (\re \alpha(H)+\delta(H)) $. Then we have
	\begin{eqnarray*}
		(\V(\calA,R))^{-N} & = & \pi^{*}O_{\PP^{n}}(-Nd(\calA))\otimes O(N\sum\limits_{\calF \in \IF(\calA)^{0}}d(\calF)E(\calF)) \otimes O(N\sum\limits_{\calF \in R}E(\calF))  \\
		&  = & O\left(N \left(\sum\limits_{H \in \calA}(\re \alpha(H)+\delta(H))\pi^{*}H+\sum\limits_{\calF \in \IF(\calA)^{0}}d(\calF)E(\calF)+\sum\limits_{\calF \in R}E(\calF)\right)\right) \\
		& = & O\left(N \left(\sum\limits_{\calF \in \IF(\calA)^{0}}\left(d(\calF)+\sum\limits_{H \in \calA}(\re \alpha(H)+\delta(H))\right)E(\calF)+\sum\limits_{\calF \in R}E(\calF)\right)\right)\\
		&  = & O\left(N \left(\sum\limits_{\calF \in \IF(\calA)^{0}}\left(\re \alpha(\calF)+\sum\limits_{H \in \calA}\delta(H)\right)E(\calF)+\sum\limits_{\calF \in R}E(\calF)\right)\right)\\
		& = & O\left(\sum\limits_{\calF \in \IF(\calA)^{0}}N\epsilon(\calF)E(\calF)\right).
	\end{eqnarray*}
	
	Since $ 0<N\epsilon(\calF)<N $ for any $ \calF \in \IF(\calA)^{0} $, by \cite[Theorem 6.2]{esnault1992lectures} we have for any $ p+q \neq n $,
	\begin{equation*}
 H^{p}(Y(\calA),\Omega_{Y(\calA)}^{q}(\log E(\calA))\otimes \V(\calA,R)) = 0.
	\end{equation*} 
	
	So by the filtered complex spectral squence of $ \Omega(\calA;R) $ we have for any $ p < n $,
	\begin{equation*}
		 H^{p}(\calA,R) =  \HH^{p}(Y(\calA),\Omega^{*}_{Y(\calA)}(\log E(\calA))\otimes \V(\calA,R)) = 0.
	\end{equation*}
\end{proof}

\begin{cor}\label{cor: n-vanishing for R(H)}
	 Suppose that $ \calA \neq \emptyset $ and $ \calA \notin \RF(\calA,\LL) $. Fix a hyperplane $ H \in \calA $. Denote by
	\begin{equation*}
		R(H) = \{\calF \in \IF(\calA)^{0} \mid H \in \calF \text{ and } \re \alpha(\calF) = 0\} .
	\end{equation*}
	Then $ (\calA,R(H)) $ is $ n $-vanishing.
\end{cor}

\begin{proof}
	 Define $ \delta \colon \calA \rightarrow \QQ $ by $ \delta(H') = 1 $ for $ H' \neq H $ and $ \delta(H) = 1-\#\calA $. Then for any subset $ S \subseteq \calA $,
	 \begin{equation*}
	 	\sum\limits_{H' \in S}\delta(H') = \begin{cases}
	 		\# S, & H \notin S, \\
	 			\# S - \#\calA,  & H \in S.
	 	\end{cases}
	 \end{equation*}
	 In particular, since $ \calA \notin \RF(\calA,\LL) $, the map $ \delta $ satisfies the condition (1) and (2) in Lemma \ref{lem: vanishing lemma}, and
	 \begin{equation*}
	 	R(H) = \{\calF \in \IF(\calA)^{0} \mid H \in \calF \text{ and } \re \alpha(\calF) = 0 \} = \{\calF \in \IF(\calA)^{0} \mid \sum\limits_{H' \in \calF}\delta(H')< 0 \text{ and } \re \alpha(\calF) = 0 \}.
	 \end{equation*}
	 So the conclusion holds by Lemma \ref{lem: vanishing lemma}.
\end{proof}

Applying the long exact sequence in Lemma \ref{lem: long exact sequences} to the pair $ (\calA,R(H)) $ defined in Corollary \ref{cor: n-vanishing for R(H)}, we have the following $ k $-vanishing criterion.

\begin{cor}\label{cor: k-vanishing given by R(H)}
		Suppose that $ \calA \neq \emptyset $. Let $ k \leq n $ be an integer. Fix a hyperplane $ H \in \calA $. If for any resonant flat $ \calF $ containing $ H $, $ \calA^{\calF} $ is $ (k+1-r(\calF)) $-vanishing, then $ \calA $ is $ k $-vanishing.
\end{cor}

\begin{proof}
	If $ \calA \in \RF(\calA,\LL) $, then by assumption we have $ \calA^{\calA} $ is $ (k+1-r(\calA)) $-vanishing. Note that $ \calA^{\calA} = \emptyset $. So we have $ r(\calA)-1 \geq k $. Note that $ M(\calA) = M(\calA_{\calA}) \times \CC^{n+1-r(\calA)} $ and $ \calA^{\calF} = (\calA_{\calA})^{\calF_{\calA}} $ for any flat $ \calF $ of $ \calA $. So we can replace $ \calA $ by $ \calA_{\calA} $ and assume that $ \calA \notin \RF(\calA,\LL) $.
	
	Let $ \calF_{1},\cdots,\calF_{m} $ be all resonant flats of $ \calA $ containing $ H $, where $ r(\calF_{1}) \leq \cdots \leq r(\calF_{m}) $. For any $ 0 \leq j \leq m $, define
	\begin{equation*}
		 R_{j} = \{\calF_{i} \mid 1 \leq i \leq j \} \cup \{\calF \in \IF(\calA)^{0} \mid H \in \calF,\ \re \alpha(\calF) = 0 \text{ and } \alpha(\calF) \neq 0 \}.
	\end{equation*}
 We claim that $ (\calA,R_{j}) $ is $ k $-vanishing for any $ j $.
	
	We prove the claim by induction on $ j $. By Corollary \ref{cor: n-vanishing for R(H)}, we have $ (\calA,R_{m}) $ is $ k $-vanishing. Suppose that $ (\calA,R_{j}) $ is $ k $-vanishing for a given $ 1 \leq j \leq m $. By definition, for any $ 1 \leq i < j  $, we have $ \calF_{i} \cap \calF_{j} \supseteq \{H\} \neq \emptyset $ and $ \calF_{i} \nsupseteq \calF_{j} $. So $ R_{j-1} \cap \RF(\calA,\LL)^{\calF_{j}} = \emptyset $. So by Lemma \ref{lem: long exact sequences} we have two long exact sequences:
		\begin{equation*}
		\begin{split}
			& \cdots \longrightarrow H^{p}(\calA,R_{j-1},\{\calF_{j}\}) \longrightarrow H^{p}(\calA,R_{j-1}) \longrightarrow \bigoplus\limits_{q+r=p-1}H^{q}(\calA^{\calF_{j}}) \otimes H^{r}(\calA_{\calF_{j}},(R_{j-1})_{\calF_{j}}) \\ 
			& \longrightarrow H^{p+1}(\calA,R_{j-1},\{\calF_{j}\}) \longrightarrow H^{p+1}(\calA,R_{j-1}) \longrightarrow \cdots
		\end{split}
	\end{equation*}
	and
	\begin{equation*}
		\begin{split}
			& \cdots \longrightarrow H^{p}(\calA,R_{j}) \longrightarrow H^{p}(\calA,R_{j-1},\{\calF_{j}\}) \longrightarrow \bigoplus\limits_{q+r=p}H^{q}(\calA^{\calF_{j}}) \otimes H^{r}(\calA_{\calF_{j}},(R_{j-1})_{\calF_{j}})\\ 
			& \longrightarrow H^{p+1}(\calA,R_{j}) \longrightarrow H^{p+1}(\calA,R_{j-1},\{\calF_{j}\}) \longrightarrow \cdots
		\end{split}
	\end{equation*}

	 Note that $ (\calA,R_{j}) $ is $ k $-vanishing and $ \calA^{\calF_{j}} $ is $ (k+1-r(\calF_{j})) $-vanishing. Furthermore, since $ \calF_{i} \nsubseteq \calF_{j} $ for any $ i > j $, we have
	\begin{equation*}
		(R_{j-1})_{\calF_{j}} = (R(H)\setminus\{R_{j}\})_{\calF_{j}} = \{\calG \in \IF(\calA_{\calF_{j}})^{0} \mid \PP(V(H)/V(\calF_{j})) \in \calG \text{ and } \re \alpha(\calG) = 0 \}.
	\end{equation*}
	Note that $ \calA_{\calF} \notin \RF(\calA_{\calF},\LL) $ for any resonant flat $ \calF \in \RF(\calA,\LL) $. So by Corollary \ref{cor: n-vanishing for R(H)} we have $ (\calA_{\calF_{j}},(R_{j-1})_{\calF_{j}}) $ is $ (r(\calF_{j})-1) $-vanishing. So the long exact sequences above imply that $ (\calA,R_{j-1}) $ is $ k $-vanishing. Then the claim follows by induction. 
	
	In particular, the pair $ (\calA,R_{0}) $ is $ k $-vanishing. Note that $ R_{0} \cap \RF(\calA,\LL) = \emptyset $. So by Theorem \ref{thm: algebraic de rham thm} we have $ H^{p}(M(\calA),\LL) = 0 $ for any $ p < k $, which implies that $ \calA $ is $ k $-vanishing. 
\end{proof}

In the end of this subsection, we consider the case when $ \calA $ contains a hyperplane $ H $ with $ m(H) = 1 $. In this case $ \LL $ can be canonically extended to a local system on $ M(\calA \setminus \{H\}) $, which we still denote by $ \LL $. The following lemma shows how the vanishing property propagates when we delete this hyperplane $ H $.

\begin{lem}\label{lem: add a hyperplane}
	Let $ k \leq n $ be a positive integer. Suppose that $ \calA $ contains a hyperplane $ H $ with $ m(H) = 1 $. If $ \calA^{H} $ is $ k $-vanishing, then the restriction map
	\begin{equation*}
		H^{p}(M(\calA\setminus \{H\}),\LL) \longrightarrow H^{p}(M(\calA),\LL)
	\end{equation*}
	is an isomorphism for any $ 0 \leq p \leq k $ and an injection for $ p = k+1 $. 
\end{lem}

\begin{proof}
	 By Theorem \ref{thm: algebraic de rham thm}, we have $ H^{p}(M(\calA\setminus \{H\}),\LL) \simeq H^{p}(\calA,\emptyset,\{H\}) $ and $ H^{p}(M(\calA),\LL) \simeq H^{p}(\calA) $. Furthermore, under these identifications, the map $ H^{p}(M(\calA\setminus \{H\}),\LL) \rightarrow H^{p}(M(\calA),\LL) $ is exactly the restriction map $ H^{p}(\calA,\emptyset,\{H\}) \rightarrow H^{p}(\calA) $. Hence the conclusion follows from the following special case of the long exact sequence (\ref{eqn: long exact sequence 1}) in Lemma \ref{lem: long exact sequences}:
	\begin{equation*}
		\cdots \longrightarrow H^{p}(\calA,\emptyset,\{H\}) \longrightarrow H^{p}(\calA) \longrightarrow H^{p-1}(\calA^{H})  \longrightarrow H^{p+1}(\calA,\emptyset,\{H\}) \longrightarrow H^{p+1}(\calA) \longrightarrow \cdots
	\end{equation*}
\end{proof}

On the other hand, restricting to $ H $ links higher‑dimensional arrangements to lower ones, as shown in the following lemma.

\begin{lem}\label{lem: dimension shift}
	Let $ k \leq n $ be a positive integer. Suppose that $ \calA $ contains a hyperplane $ H $ with $ m(H) = 1 $. If $ \calA^{\calF} $ is $ (k+1-r(\calF)) $-vanishing for any resonant flat $ \calF $ properly containing $ H $, then the restriction map
	\begin{equation*}
		H^{p}(M(\calA\setminus \{H\}),\LL) \longrightarrow H^{p}(M(\calA^{H}),\LL)
	\end{equation*}
	is an isomorphism for any $ 0 \leq p \leq k-2 $ and an injection for $ p = k-1 $. 
\end{lem}

\begin{proof}
		Similar to the proof of Corollary \ref{cor: k-vanishing given by R(H)}, we may assume that $ \calA \notin \RF(\calA,\LL) $. Define
		\begin{equation*}
			R = \{\calF \in \IF(\calA)^{0} \mid H \in \calF,\ \re \alpha(\calF) = 0 \text{ and } \alpha(\calF) \neq 0 \}.
		\end{equation*}
		 By the same induction in the proof of Corollary \ref{cor: k-vanishing given by R(H)}, we obtain that $ (\calA,R \cup \{H\}) $ is $ k $-vanishing. Since $ R \cap \RF(\calA,\LL) = \emptyset $, by Theorem \ref{thm: algebraic de rham thm} we have $ H^{p}(M(\calA\setminus \{H\}),\LL) \simeq H^{p}(\calA,R,\{H\}) $ and $ H^{p}(M(\calA^{H}),\LL) \simeq H^{p}(\calA^{H}) $. Furthermore, under these identifications, the map $ H^{p}(M(\calA\setminus \{H\}),\LL) \rightarrow H^{p}(M(\calA^{H}),\LL) $ is exactly the restriction map $ H^{p}(\calA,R,\{H\}) \rightarrow H^{p}(\calA^{H}) $. Hence the conclusion follows from the following special case of the long exact sequence (\ref{eqn: long exact sequence 2}) in Lemma \ref{lem: long exact sequences}:
		\begin{equation*}
			\cdots \longrightarrow H^{p}(\calA,R \cup \{H\}) \longrightarrow H^{p}(\calA,R,\{H\}) \longrightarrow H^{p}(\calA^{H})  \longrightarrow H^{p+1}(\calA,R \cup \{H\}) \longrightarrow \cdots
		\end{equation*}
\end{proof}

\subsection{Proof of Theorem \ref{main: n-vanishing} and its Application to Line Arrangements}\label{subsection: n-vanishing}

By Lemma \ref{lem: vanishing lemma}, to prove that $ \calA $ is $ n $-vanishing, it suffices to find a map $ \delta \colon \calA \rightarrow \QQ $ such that
\begin{enumerate}
	\item $ \sum\limits_{H \in \calA}\delta(H) = 0 $, and
	\item $ \sum\limits_{H \in \calF}\delta(H) > 0 $ for any $ \calF \in \RF(\calA,\LL) $.
\end{enumerate}

The following lemma gives an equivalent condition for such a map $ \delta $ to exist.

\begin{lem}\label{lem: non-negative linear combination}
	Let $ W $ be a finite dimensional $ \QQ $-vector space. Let $ f_{1},\cdots,f_{m} \in W^{*} $ be $ m $ linear functionals on $ W $. Then the following are equal:
	\begin{enumerate}
		\item For any $ w \in W $, there exists $ 1 \leq i \leq m $ such that $ f_{i}(w) \leq 0 $.
		\item There exist natural numbers $ \lambda_{1},\cdots,\lambda_{m} $, not all zero, such that $ \sum_{i=1}^{m} \lambda_{i}f_{i} = 0 $.
	\end{enumerate}
\end{lem}

\begin{proof}
	Sufficiency is trivial. We prove the necessity by induction on $ m $. The conclusion obviously holds when $ m = 1 $. 
	
	Suppose that the conclusion holds for $ m-1 $. Then for $ m $, we may assume that $ f_{m} \neq 0 $. If for any $ w \in W $, there exists $ 1 \leq i \leq m - 1 $ such that $ f_{i}(w) \leq 0 $, then by the induction hypothesis, there exist natural numbers $ \lambda_{1},\cdots,\lambda_{m-1} $, not all zero, such that $ \sum_{i=1}^{m-1} \lambda_{i}f_{i} = 0 $, which implies that the conclusion holds. 
	
	So we may assume that there exists $ w_{0} \in W $ such that $ f_{i}(w_{0}) > 0 $ for all $ 1 \leq i \leq m-1 $. By assumption we have $ f_{m}(w_{0}) \leq 0 $. Since $ f_{m} \neq 0 $, By slightly perturbing $ w_{0} $, we may assume $ f_{m}(w_{0}) < 0 $. By rescaling we may further assume $ f_{i}(w_{0}) \in \ZZ $ for all $ 1 \leq i \leq m $.
	
	Consider the $ \QQ $-vector space $ W' = W/\Span(w_{0}) $. For each $ 1 \leq i \leq m-1 $, the functional $ f_{i} - \frac{f_{i}(w_{0})}{f_{m}(w_{0})}f_{m} $ maps $ w_{0} $ to $ 0 $ and thus defines a linear functional on $ W' $, which we denote by $ f'_{i} $. If there exists some $ w \in W $ such that $ f'_{i}([w]) > 0 $ for all $ 1 \leq i \leq m-1 $, we can choose $ \mu \in \QQ $ such that
	\begin{equation*}
	0 < f(w+\mu w_{0}) = f_{m}(w)+\mu f_{m}(w_{0}) < \min_{1 \leq i \leq m-1}\left(-\frac{f_{m}(w_{0})}{f_{i}(w_{0})}f'_{i}([w])\right).
	\end{equation*}
	Then we have $ f_{m}(w+\mu w_{0}) > 0 $ and for any $ 1 \leq i \leq m-1 $,
	\begin{equation*}
		f_{i}(w+\mu w_{0}) = f'_{i}([w]) + \frac{f_{i}(w_{0})}{f_{m}(w_{0})}f_{m}(w+\mu w_{0}) > 0,
	\end{equation*}
	which contradicts the assumption.
	
	Therefore, for any $ w \in W $, there exists $ 1 \leq i \leq m-1 $ such that $ f'_{i}([w]) \leq 0 $. By the induction hypothesis there exist natural numbers $ \lambda'_{1},\cdots,\lambda'_{m-1} $, not all zero, such that $ \sum_{i=1}^{m-1} \lambda'_{i}f'_{i} = 0 $. Thus we have
	\begin{equation*}
		\sum\limits_{i=1}^{m-1}(-\lambda'_{i}f_{m}(w_{0}))f_{i}+(\sum\limits_{i=1}^{m-1}\lambda'_{i}f_{i}(w_{0}))f_{m} = -f_{m}(w_{0})\sum\limits_{i=1}^{m-1}\lambda'_{i}(f_{i} - \frac{f_{i}(w_{0})}{f_{m}(w_{0})}f_{m}) = 0.
	\end{equation*}
	Hence the conclusion holds by induction.
\end{proof}

Now we are ready to prove Theorem \ref{main: n-vanishing}.

\begin{proof}[Proof of Theorem \ref{main: n-vanishing}]
	Take
	\begin{equation*}
		W = \{ \delta \colon \calA \rightarrow \QQ  \mid  \sum_{H \in \calA}\delta(H) = 0 \} 
	\end{equation*}
	and for any $ \calF \in \RF(\calA,\LL) $, define
	\begin{equation*}
		f(\calF) \colon W \rightarrow \QQ,\ f(\calF)(\delta) = \sum\limits_{H \in \calF} \delta(H).
	\end{equation*}
	
	For any non-identically-zero map $ \lambda \colon \RF(\calA,\LL) \rightarrow \NN $, by assumption, there exist $ H_{1},H_{2} \in \calA $ such that 
	\begin{equation*}
	\sum\limits_{\calF \ni H_{1}}\lambda(\calF) \neq \sum\limits_{\calF \ni H_{2}}\lambda(\calF).
	\end{equation*}
	
	Define $ \delta_{\lambda} \in W $ as
	\begin{equation*}
		\delta_{\lambda}(H) = \begin{cases}
			1, & H = H_{1}, \\
			-1, & H = H_{2}, \\
			0, & \text{otherwise}.
		\end{cases}
	\end{equation*}
	
	Then we have
	\begin{equation*}
		\sum\limits_{\calF \in \RF(\calA,\LL)} \lambda(\calF)f(\calF)(\delta_{\lambda}) = \sum\limits_{H \in \calA}(\sum\limits_{\calF \ni H}\lambda(\calF))\delta_{\lambda}(H) = 		\sum\limits_{\calF \ni H_{1}}\lambda(\calF) - \sum\limits_{\calF \ni H_{2}}\lambda(\calF) \neq 0.
	\end{equation*}
	
	 So by Lemma \ref{lem: non-negative linear combination}, there exists $ \delta \in W $ such that for all $ \calF \in \RF(\calA,\LL) $, 
	 \begin{equation*}
	 	\sum\limits_{H \in \calF} \delta(H) = f(\calF)(\delta) > 0.
	 \end{equation*}
	 
	 By slightly perturbing $ \delta $, we may assume that $ \sum\limits_{H \in \calF}\delta(H) \neq 0 $ for any $ \calF \in \IF(\calA)^{0} $ such that $ \re \alpha(\calF) = 0 $. Then taking
	 \begin{equation*}
	 	R = \{\calF \in \IF(\calA)^{0} \mid  \sum\limits_{H \in \calF}\delta(H)< 0 \text{ and } \re \alpha(\calF) = 0 \},
	 \end{equation*}
	 by Lemma \ref{lem: vanishing lemma} we have $ (\calA,R) $ is $ n $-vanishing, which implies that $ H^{p}(M(\calA),\LL) = 0 $ for all $ p \neq n $ since $ R \cap \RF(\calA,\LL) = \emptyset $.
\end{proof}

As an application, we use this theorem to prove Theorem \ref{main: one point}.

\begin{cor}[Theorem \ref{main: one point}]
	Let $ \calA $ be a line arrangement in $ \PP^{2} $ and $ \LL $ a complex rank-one local system on the complement $ M(\calA) $ with monodromy map $ m \colon \calA \rightarrow \CC^{\times} $. If there exists a resonant flat $ \calF $ of rank $ 2 $ such that
	\begin{enumerate}
		\item there exists $ H_{1} \in \calF $ such that $ m(H_{1}) \neq 1 $ and no other resonant flat contains $ H_{1} $, and
		\item  there exists $ H_{2} \in \calA \setminus \calF $ such that $ m(H_{2}) \neq 1 $,
	\end{enumerate}
	then $ \calA $ is $ 2 $-vanishing.
\end{cor}

\begin{proof}
	We prove the corollary by contradiction. Suppose that $ \calA $ is not $ 2 $-vanishing. Then $ H^{1}(M(\calA),\LL) \neq 0 $. For any flat $ \calG \in L(\calA) $ of rank $ 2 $ such that $ H_{1} \notin \calG $, there exists a unique line in $ \PP^{2} $ passing through the points $ Z(\calF) $ and $ Z(\calG) $, which we denote by $ H_{\calG} $. Define
	\begin{equation*}
		\calA' = \calA \cup \{H_{\calG} \mid \calG \in L(\calA),\ r(\calG) = 2,\ H_{1} \notin \calG\}
	\end{equation*}
	and denote
	\begin{equation*}
		\calF' = \{H \in \calA' \mid Z(\calF) \in H\} = \calF \cup \{H_{\calG} \mid \calG \in L(\calA),\ r(\calG) = 2,\ H_{1} \notin \calG\}.
	\end{equation*}
	By definition, $ \calF' $ is the only resonant flat of $ \calA' $ containing $ H_{1} $. Furthermore, for any flat $ \calG' \in L(\calA') $ of rank $ 2 $ such that $ H_{1} \notin \calG' $, the arrangement $ \calA' $ must contain the line passing through $ Z(\calF) $ and $ Z(\calG') $, or equivalently, $ \calG' \cap \calF' \neq \emptyset$

	By Lemma \ref{lem: add a hyperplane}, we have the restriction map $ H^{1}(M(\calA),\LL) \rightarrow H^{1}(M(\calA'),\LL) $ is injective. So the pair $ (\calA',\LL) $ is resonant. By Theorem \ref{main: n-vanishing}, there exists a non-identically-zero map $ \lambda \colon \RF(\calA',\LL) \rightarrow \NN $ such that for any $ H \in \calA' $,
	\begin{equation*}
		\sum\limits_{\calG' \ni H}\lambda(\calG') = \sum\limits_{\calG' \ni H_{1}}\lambda(\calG') = \lambda(\calF').
	\end{equation*}
	
	 For any resonant flat $ \calG' \in \RF(\calA,\LL) $ containing $ H_{2} $, we have $ H_{1} \notin \calG' $ and $ r(\calG') = 2 $, which implies that $ \calG' \cap \calF' \neq \emptyset $. Taking $ H \in \calG' \cap \calF' $ in the above equality, we have $ \lambda(\calG') = 0 $. 
	 
	 Hence for any $ H \in \calA' $,
	\begin{equation*}
		\sum\limits_{\calG' \ni H}\lambda(\calG') = \sum\limits_{\calG' \ni H_{2}}\lambda(\calG') = 0,
	\end{equation*}
	which contradicts the assumption that $ \lambda $ is non-identically-zero.
\end{proof}

\subsection{A Restriction Argument}\label{subsection: restriction}

In general, there may not exist a hyperplane in $ \calA $ with trivial monodromy. However, we can always add a hyperplane $ H_{0} \notin \calA $ into $ \calA $. Then the monodromy of $ H_{0} $ is naturally $ 1 $, and we can use Lemma \ref{lem: dimension shift}. The following lemma compares the combinatorial structure of $ \calA $ with that of its restriction to the hyperplane $ H_{0} $.

\begin{lem}\label{lem: relations between flats after restriction}
	Let $ H_{0} \notin \calA $ be a hyperplane in $ \PP^{n} $. Define $ \calA_{0} = \{H \cap H_{0} \mid H \in \calA\}  $. Then $ \calA_{0} $ is a hyperplane arrangement in $ H_{0} \simeq \PP^{n-1} $ with $ M(\calA_{0}) \subseteq M(\calA) $. Furthermore, $ \calA_{0} $ satisfies the following properties.
	\begin{enumerate}
		\item For any flat $ \calG \in L(\calA_{0}) $, $ \widetilde{\calG} = \{H \in \calA \mid H \cap H_{0} \in \calG\} $ is a flat of $ \calA $ such that
		\begin{equation}\label{eqn: rank formula for section and lifting}
			   r(\widetilde{\calG}) = \begin{cases}
				r(\calG) + 1, & Z(\widetilde{\calG}) \subseteq H_{0}, \\
				r(\calG), & \text{otherwise}.
			\end{cases}
		\end{equation}	

		\item For any flat $ \calF \in L(\calA) $, if $ \calF \cup \{H_{0}\} \in L(\calA \cup \{H_{0}\}) $, or equivalently,
		\begin{equation}\label{eqn: condition for a flat to descend}
			\calF = \{H \in \calA \mid Z(\calF)\cap H_{0} \subseteq H\},
		\end{equation}
		then $ \calF_{0} = \{H \cap H_{0} \mid H \in \calF\} $ is a flat of $ \calA_{0} $. Furthermore, the map $ \calF \mapsto \calF_{0} $ induces a bijection
		\begin{equation*}
		\{\calF \in L(\calA) \mid 	\calF \text{ satisfies the condition }(\ref{eqn: condition for a flat to descend})\} \xrightarrow{1:1} L(\calA_{0}),
		\end{equation*}
		with inverse given by $ \calG \mapsto \widetilde{\calG} $.
	\item
	Let $ \calF $ be a nonempty subset of $ \calA $. Then $ \calF \cup \{H_{0}\} \in \IF(\calA \cup \{H_{0}\}) $ if and only if  $ \calF $ is a flat of $ \calA $ such that $ Z(\calF) \subseteq H_{0} $ and $ Z(\widetilde{\calG}) \subseteq H_{0} $ for any irreducible component $ \calG $ of $ \calF_{0} $.
	\end{enumerate}
\end{lem}

\begin{proof}
	 Since $ H_{0} \notin \calA $, we have $ \calA_{0} = (\calA \cup \{H_{0}\})^{H_{0}}  $ is a hyperplane arrangement in $ H_{0} $ with $ M(\calA_{0}) \subseteq M(\calA) $. 
	 
	 (1): For any flat $ \calG $ of $ \calA_{0} $,  $ \widetilde{\calG} = \{H \in \calA \mid Z(\calG) \subseteq H \} $ is a flat of $ \calA $, and $ Z(\widetilde{\calG}) \cap H_{0} = Z(\calG) $. Then the rank formula (\ref{eqn: rank formula for section and lifting}) follows directly from the following equality
	 	\begin{equation*}
	 		\dim Z(\widetilde{\calG}) = \begin{cases}
	 		\dim Z(\calG),  & Z(\widetilde{\calG}) \subseteq H_{0}, \\
	 		\dim Z(\calG)+1, & \text{otherwise}.
	 	\end{cases}
	 \end{equation*}
	
	(2): For any flat $ \calF $ of $ \calA $ satisfying the condition (\ref{eqn: condition for a flat to descend}), we have
	\begin{equation*}
		\calF_{0} = \{H \cap H_{0} \mid H \in \calA, Z(\calF) \cap H_{0} \subseteq H \} = \{H \cap H_{0} \mid H \in \calA, Z(\calF) \cap H_{0} \subseteq H \cap H_{0}\}
	\end{equation*}
	is a flat of $ \calA_{0} $. Conversely, for any flat $ \calG $ of $ \calA_{0} $, we have
	\begin{equation*}
		\widetilde{\calG} = \{H \in \calA \mid Z(\calG) \subseteq H \cap H_{0} \} = \{H \in \calA \mid Z(\widetilde{\calG}) \cap H_{0} \subseteq H \}
	\end{equation*}
	satisfies the condition (\ref{eqn: condition for a flat to descend}). Then the conclusion follows from 
	\begin{equation*}
		(\widetilde{\calG})_{0} = \{H \cap H_{0} \mid H \in \calA, H \cap H_{0} \in \calG\} = \calG
	\end{equation*}
	and
	\begin{equation*}
		\widetilde{\calF_{0}} = \{H \in \calA \mid Z(\calF_{0})  \subseteq H \cap H_{0} \} = \{H \in \calA \mid Z(\calF) \cap H_{0} \subseteq H \}  = \calF.
	\end{equation*}

	(3): Necessity: Since $ \calF \cup \{H_{0}\} $ is a flat of $ \calA \cup \{H_{0}\} $, $ \calF = \{H \in \calA \mid Z(\calF) \cap H_{0} \subseteq H\} $ is a flat of $ \calA $. Since $ \calF \cup \{H_{0}\} $ is irreducible, $ \calF $ is not a flat of $ \calA \cup\{H_{0}\} $, which implies that $ Z(\calF) \subseteq H_{0} $. In particular, $ \calF $ satisfies the condition (\ref{eqn: condition for a flat to descend}), and $ \calF_{0} $ is a nonempty flat of $ \calA_{0} $. If there exists an irreducible component $ \calG_{1} $ of $ \calF_{0} $ such that $ Z(\widetilde{\calG}) \nsubseteq H_{0} $, let $ \calG_{2} $ be the union of the remaining irreducible components of $ \calF_{0} $. Then 	
	\begin{equation*}
		r(\calG_{1})+r(\calG_{2}) = r(\calF_{0}),\ \calG_{1} \sqcup \calG_{2} = \calF_{0}.
	\end{equation*}
	By definition, $ \widetilde{\calG_{1}} \cup \widetilde{\calG_{2}}  = \calF $ and $ r(\widetilde{\calG_{1}}) + r(\widetilde{\calG_{2}}) \leq 	r(\calG_{1})+r(\calG_{2}) + 1 = r(\calF_{0})+1 = r(\calF) \leq r(\widetilde{\calG_{1}}) + r(\widetilde{\calG_{2}}) $. So $  r(\widetilde{\calG_{2}}) = r(\calG_{2}) + 1 $. So $ Z(\widetilde{\calG_{2}}) \subseteq H_{0} $. So $ \widetilde{\calG_{2}} \cup \{H_{0}\} $ is a flat of $ \calA \cup \{H_{0}\} $. Note that 
	\begin{equation*}
		r(\widetilde{\calG_{1}}) + r(\widetilde{\calG_{2}} \cup \{H_{0}\}) = r(\widetilde{\calG_{1}}) + r(\widetilde{\calG_{2}}) =  r(\calF) = r(\calF \cup \{H_{0}\}),\  \widetilde{\calG_{1}} \sqcup (\widetilde{\calG_{2}} \cup \{H_{0}\})= \calF\cup \{H_{0}\}.
	\end{equation*}
	 This contradicts the assumption that $ \calF \cup \{H_{0}\} $ is irreducible. So $ Z(\widetilde{\calG}) \subseteq H_{0} $ for any irreducible component $ \calG $ of $ \calF_{0} $. 
	
	Sufficiency: Since $ \calF $ is a flat of $ \calA $ and $ Z(\calF) \subseteq H_{0} $, $ \calF \cup \{H_{0}\} $ is a flat of $ \calA \cup \{H_{0}\} $. Let $ \calF_{1} \cup \{H_{0}\} $ be the irreducible component of $ \calF \cup \{H_{0}\} $ containing $ H_{0} $. Let $ \calF_{2} $ be the union of the remaining irreducible components of $ \calF \cup \{H_{0}\} $. By Lemma \ref{lem: irreducible decomposition}, $ \calF_{2} \cup \{H_{0}\} $ is a flat of $ \calA \cup \{H_{0}\} $. The same holds trivially for $ \calF_{1} $. So there exist flats $ \calG_{1},\calG_{2} $ of $ \calA_{0} $ such that $ \widetilde{\calG_{i}} = \calF_{i} $. Note that $ Z(\calF_{2}) \nsubseteq H_{0} $. By the rank formula (\ref{eqn: rank formula for section and lifting}) we have
	\begin{equation*}
		r(\calG_{1})+r(\calG_{2}) = r(\widetilde{\calG_{1}} \cup \{H_{0}\})-1 + r(\widetilde{\calG_{2}})  = r(\calF_{1} \cup \{H_{0}\})-1 + r(\calF_{2})	 = r(\calF \cup \{H_{0}\})-1 = r(\calF_{0}).
	\end{equation*}
	Together with $ \calG_{1} \sqcup \calG_{2} = \calF_{0} $, this rank equality implies that $ \calG_{2} $ is a union of some irreducible components of $ \calF_{0} $. Since $ Z(\widetilde{\calG_{2}}) \nsubseteq H_{0} $ but $ Z(\widetilde{\calG}) \subseteq H_{0} $ for all irreducible components $ \calG $ of $ \calF_{0} $, we have $ \calG_{2} = \emptyset $. Hence $ \calF \cup \{H_{0}\} $ is irreducible.
\end{proof}

For any subspace $ V \subseteq \CC^{n+1} $ of codimension $ \geq 2 $, the following lemma shows the existence of a generic hyperplane containing $ \PP(V) $ with respect to $ \calA $.

\begin{lem}\label{lem: existence of general section}
	Let $ V \subseteq \CC^{n+1} $ be a subspace of codimension $ \geq 2 $. Then there exists a hyperplane $ H_{0} \supseteq \PP(V) $ in $ \PP^{n} $ such that for any flat $ \calF \in L(\calA) $, $ Z(\calF) \subseteq H_{0} $ if and only if $ Z(\calF) \subseteq \PP(V) $.
\end{lem}

\begin{proof}
	For each flat $ \calF \in L(\calA) $ such that $ Z(\calF) \nsubseteq \PP(V) $, we have
	\begin{equation*}
		\{l \in (\CC^{n+1})^{*} \mid l(V(\calF)+V) = 0 \} \subsetneq \{l \in (\CC^{n+1})^{*} \mid l(V) = 0 \}.
	\end{equation*}
	So there exists $ l_{0} \in (\CC^{n+1})^{*} $ such that $ l_{0}(V) = 0 $ but $ l_{0}(V(\calF)+V) \neq 0 $ for any flat $ \calF \in L(\calA) $ such that $ Z(\calF) \nsubseteq \PP(V) $. Then $ H_{0} = \{l_{0} = 0\} \subseteq \PP^{n} $ meets the requirement.
\end{proof}

Taking $ V = 0 $ in Lemma \ref{lem: existence of general section}, by Lemma \ref{lem: dimension shift} and \ref{lem: relations between flats after restriction} we obtain the following well-known lemma.

\begin{lem}\label{lem: dimension shift for generic section}
	When $ n \geq 2 $, there exists a hyperplane $ H_{0} $ in $ \PP^{n} $ such that $ Z(\calF) \nsubseteq H_{0} $ for any flat $ \calF \in L(\calA) $ of rank $ \leq n $. Keeping the same notation as Lemma \ref{lem: relations between flats after restriction} one has
	\begin{enumerate}
		\item $ M(\calA_{0}) \subseteq M(\calA) $ and for any $ 0 \leq p \leq n-2 $,
		\begin{equation*}
			H^{p}(M(\calA),\LL) \simeq H^{p}(M(\calA_{0}),\LL).
		\end{equation*}
		\item The map $ \calF \mapsto \calF_{0} $ induces an isomorphism of ranked posets
		\begin{equation*}
			\{\calF \in L(\calA) \mid r(\calF) \leq n-1\}	 \xrightarrow{\simeq} \{\calG \in L(\calA_{0}) \mid r(\calG) \leq n-1\}.
		\end{equation*}
	\end{enumerate}
\end{lem}

 When $ V = V(\calI) $ for some irreducible flat $ \calI \in \IF(\calA) $, we have the following corollary. 

\begin{cor}\label{cor: irreducible section and relation among flats}
	Let $ \calI \in \IF(\calA) $ be an irreducible flat of rank $ \geq 2 $. Then there exists a hyperplane $ H_{0} \supseteq Z(\calI) $ in $ \PP^{n} $ such that for any flat $ \calF \in L(\calA) $, $ Z(\calF) \subseteq H_{0} $ if and only if $ \calF \supseteq \calI $. Furthermore, with such an $ H_{0} $ fixed, keeping the same notation as Lemma \ref{lem: relations between flats after restriction}, the following statements hold:
	\begin{enumerate}
		\item The map $ \calG \mapsto \widetilde{\calG} $ induces a bijection
		\begin{equation*}
			\{\calG \in \IF(\calA_{0}) \mid \calG \supseteq \calI_{0} \} \xrightarrow{1:1} \{\calF \in \IF(\calA) \mid \calF \supseteq \calI\}.
		\end{equation*}
		Furthermore, $ r(\widetilde{\calG}) = r(\calG) + 1 $ for any irreducible flat $ \calG \in \IF(\calA_{0}) $ containing $ \calI_{0} $. 
		\item The map $ \calF \mapsto \calF \cup \{H_{0}\} $ induces a rank-preserving bijection
		\begin{equation*}
			\{\calF \in \IF(\calA) \mid \calF \supseteq \calI \} \xrightarrow{1:1} \{\calF' \in \IF(\calA \cup \{H_{0}\}) \mid \calF' \supsetneq \{H_{0}\}\}.
		\end{equation*}	
	\end{enumerate}
\end{cor}

\begin{proof}
	The existence of $ H_{0} $ follows directly by taking $ V = V(\calI) $ in Lemma \ref{lem: relations between flats after restriction}. We further note that for any flat $ \calG $ of $ \calA_{0} $,
	\begin{equation*}
		Z(\widetilde{G}) \subseteq H_{0} \iff Z(\widetilde{G}) \subseteq Z(\calI) \iff \widetilde{\calG} \supseteq \calI \iff \calG \supseteq \calI_{0}. 
	\end{equation*}

	(1): We first verify that this map is well-defined. For any irreducible flat $ \calG $ of $ \calA_{0} $ containing $ \calI_{0} $, $ \widetilde{G} $ is a flat of $ \calA $ containing $ \calI $. Let $ \calF_{1} $ be the irreducible component of $ \calA $ containing $ \calI $. Let $ \calF_{2} $ be the union of the remaining irreducible components of $ \widetilde{\calG} $. Then
	\begin{equation*}
		\calF_{1} \supseteq \calI,\ \calF_{2} \cap \calI = \emptyset,\ r(\calF_{1})+r(\calF_{2}) = r(\widetilde{\calG}),\ \calF_{1} \sqcup \calF_{2} = \widetilde{\calG}.
	\end{equation*}
	
	We further denote by $ \calF_{3} = \{H \in \calA \mid Z(\calF_{2}) \cap H_{0} \subseteq H \} $. Then $ Z(\calF_{2}) \cap H_{0} \subseteq Z(\calF_{3}) \subseteq Z(\calF_{2}) $. If $ Z(\calF_{3}) = Z(\calF_{2}) \cap H_{0} \subseteq H_{0} $, then $ \calF_{3} \supseteq \calI $ and $ Z(\calF_{3}) \subseteq Z(\calI \cup \calF_{2}) $. However, since $ r(\calF_{1}) + r(\calF_{2}) = r(\calF_{1}\cup\calF_{2}) $, we have $ \CC^{n+1} = V(\calF_{1})+V(\calF_{2}) = V(\calI) + V(\calF_{2}) $, which implies that $ \dim V(\calF_{2})/V(\calI \cup \calF_{2}) = \codim V(\calI) = r(\calI) \geq 2 $. So $ Z(\calI \cup \calF_{2}) \subsetneq Z(\calF_{2}) \cap H_{0} $, a contradiction. So we have $ Z(\calF_{3}) = Z(\calF_{2}) $, or equivalently, $ \calF_{2} $ satisfies the condition (\ref{eqn: condition for a flat to descend}). So by Lemma \ref{lem: relations between flats after restriction}(2), there exist flats $ \calG_{1},\calG_{2} $ of $ \calA_{0} $ such that $ \widetilde{\calG_{i}} = \calF_{i} $. By definition we have $ \calG = \calG_{1} \cup \calG_{2} $. By the rank formula (\ref{eqn: rank formula for section and lifting}) we have $ r(\calG) = r(\widetilde{\calG}) - 1 = (r(\calF_{1})-1)+r(\calF_{2}) = r(\calG_{1}) + r(\calG_{2}) $. Since $ \calG $ is irreducible, we have $ \calG_{2} = \emptyset $, which implies that $ \widetilde{\calG} = \calF_{1} $ is irreducible.
	
	Now it suffices to prove that this map is surjective.  For any irreducible flat $ \calF $ of $ \calA $ containing $ \calI $, $ \calF_{0}$ is a flat of $ \calA_{0} $ containing $ \calI_{0} $. Let $ \calG_{1} $ be the irreducible component of $ \calF_{0} $ containing $ \calI_{0} $. Let $ \calG_{2} $ be the union of the remaining irreducible components of $ \calF_{0} $. Then 
		\begin{equation*}
		\calG_{1} \supseteq \calI_{0},\ \calG_{2} \cap \calI_{0} = \emptyset,\ r(\calG_{1})+r(\calG_{2}) = r(\calF_{0}),\ \calG_{1} \sqcup \calG_{2} = \calF_{0}.
	\end{equation*}
	So $ \widetilde{\calG_{1}} \cup \widetilde{\calG_{2}} = \widetilde{\calF_{0}} = \calF $, and $ r(\widetilde{\calG_{1}}) + r(\widetilde{\calG_{2}}) =  r(\calG_{1})+1+r(\calG_{2}) = r(\calF_{0}) + 1 = r(\calF) $. Since $ \calF $ is irreducible, we have $ \widetilde{\calG_{2}} = \emptyset $, which implies that $ \calF_{0} = \calG_{1} $ is irreducible. So the map is a bijection. The rank relation follows directly from the rank formula (\ref{eqn: rank formula for section and lifting}).
	
	(2): By Lemma \ref{lem: relations between flats after restriction}(3), it suffices to prove that for a flat $ \calF $ of $ \calA $ containing $ \calI $, $ \calF $ is irreducible if and only if $ \calG \supseteq \calI_{0} $ for any irreducible component $ \calG $ of $ \calF_{0} $. This follows directly from part (1), since the latter is equivalent to $ \calF_{0} $ being irreducible.
\end{proof}

In light of Lemma \ref{lem: dimension shift for generic section}, one may conjecture that $ h^{p}(M(\calA),\LL) $ depends only on flats of rank $ \leq p+1 $. However, the following proposition reveals that higher-rank flats also affect low-degree cohomology.

\begin{prop}\label{prop: k-vanishing given by higher ranked resonant flats}
	Let $ \calA $ be a hyperplane arrangement in $ \PP^{n} $. If there exists $ \calI \in \IF(\calA) $ such that $ \calF \nsubseteq \calI $ for any $ \calF \in \RF(\calA,\LL) $, then $ \calA $ is $ (n+1-r(\calI)) $-vanishing.
\end{prop}

\begin{proof}
	We prove the conclusion by induction on $ r(\calI) $.
	
	When $ r(\calI) = 1 $, we have $  \calI=\{H\} $ for some $ H \in \calA $. Then $ H \notin \calF $ for any $ \calF \in \RF(\calA,\LL) $. So by Corollary \ref{cor: n-vanishing for R(H)} we have $ \calA $ is $ n $-vanishing.
	
	Suppose the conclusion holds when $ r(\calI) < k $. When $ r(\calI) = k > 1 $, by Corollary \ref{cor: irreducible section and relation among flats}, there exists a hyperplane $ H_{0} \supseteq Z(\calI) $ in $ \PP^{n} $ such that for any flat $ \calF \in L(\calA) $, $ Z(\calF) \subseteq H_{0} $ if and only if $ \calF \supseteq \calI $. We further keep the same notation as Corollary \ref{cor: irreducible section and relation among flats}. 
	
	By Corollary \ref{cor: irreducible section and relation among flats}(2), the absence of resonant flats of $ \calA $ containing $ \calI $ implies that no resonant flat of $ \calA \cup \{H_{0}\} $ with rank $ > 1 $ contains $ H_{0} $. So by Lemma \ref{lem: dimension shift}, it suffices to prove that $ \calA_{0} = (\calA \cup \{H_{0}\})^{H_{0}} $ is $ (n+1-k) $-vanishing. By Corollary \ref{cor: irreducible section and relation among flats}(1), $ \calI_{0} $ is an irreducible flat of $ \calA_{0} $ with rank $ k-1 $. Furthermore, the absence of resonant flats of $ \calA $ containing $ \calI $ implies that $ \calG \nsubseteq \calI_{0} $ for any $ \calG \in \RF(\calA_{0},\LL) $. So by the induction hypothesis, $ \calA_{0} $ is $ (n-1)+1-(k-1) = (n+1-k) $-vanishing. Hence the conclusion holds by induction.
\end{proof}

\subsection{Lifting Technique and Proof of Theorem \ref{main: 2-vanishing}}\label{subsection: lifting}

We first introduce the definition of liftings. Identify $ \PP^{n} $ with the hyperplane $ H_{0} = \{x_{n+1} = 0\} $ in $ \PP^{n+1} $. A \textbf{lifting of a hyperplane} $ H \subseteq \PP^{n} $ is a hyperplane $ \widetilde{H} \subseteq \PP^{n+1} $ with $ \widetilde{H} \cap H_{0} = H $. A \textbf{lifting of a hyperplane arrangement} $ \calA $ in $ \PP^{n} $ is a hyperplane arrangement $ \widetilde{\calA} $ in $ \PP^{n+1} $ obtained by choosing a lifting $ \widetilde{H} $ for each hyperplane $ H $ in $ \calA $. 

Fixing a lifting $ \widetilde{\calA} $ of $ \calA $, we can canonically extend $ \LL $ to a local system on $ M(\widetilde{\calA}) $ by taking $ m(\widetilde{H}) = m(H) $, which we still denote by $ \LL $. Since $ H_{0} \notin \widetilde{A} $, we are in the setting of Lemma \ref{lem: relations between flats after restriction}. Following its notation, we denote $ \widetilde{S} = \{\widetilde{H} \mid H \in S \} $ for any subset $ S \subseteq \calA $. The following lemma then follows directly from Lemma \ref{lem: dimension shift} and Lemma \ref{lem: relations between flats after restriction}.

\begin{lem}\label{lem: dimension shift for lifting}
	Let $ \widetilde{\calA} $ be a lifting of $ \calA $. 	Let $ k \leq n+1 $ be a positive integer. If for any nonempty flat $ \calF \in L(\calA) $ satisfying that
	\begin{equation}\label{eqn: condition for nonresonant lifting}
		r(\calF) \leq n,\ m(\calF) = 1, \text{ and } r(\widetilde{\calG}) = r(\calG)+1 \text{ for any irreducible component } \calG \text{ of } \calF,
	\end{equation}
	the arrangement $ \calA^{\calF} $ is $ (k-r(\calF)) $-vanishing, then the restriction map
	\begin{equation*}
		H^{p}(M(\widetilde{A}),\LL) \longrightarrow H^{p}(M(\calA),\LL)
	\end{equation*}
	is an isomorphism for any $ 0 \leq p \leq k-2 $ and an injection for $ p = k-1 $. 
\end{lem}

In particular, we focus on liftings given by a bipartition, which are constructed in the following lemma.

\begin{lem}\label{lem: lifting given by partition}
	Given a bipartition $ \calA = \calA_{1} \sqcup \calA_{2} $, there exists a lifting $ \widetilde{A} $ of $ \calA $ such that for any subset $ S \subseteq \calA $, 
	\begin{equation*}
		r(\widetilde{S}) = r(S) \iff r(S) = r(S \cap \calA_{1}) + r(S \cap \calA_{2}).
	\end{equation*}
\end{lem}

\begin{proof}
For simplicity, we denote by $ S_{i} = S \cap \calA_{i} $ for any subset $ S \subseteq \calA $. Choose $ (z_{0},\cdots,z_{n}) \in \CC^{n+1} $ such that $ (z_{0},\cdots,z_{n}) \notin V(S_{1}) + V(S_{2}) $ for any subset $ S \subseteq \calA $ with $ V(S_{1}) + V(S_{2}) \neq \CC^{n+1} $. Take
\begin{equation*}
	V(\widetilde{H}) = \begin{cases}
		V(H) + \CC \cdot (0,\cdots,0,1)  & H \in \calA_{1} \\
		V(H) + \CC \cdot (z_{0},\cdots,z_{n},1)  & H \in \calA_{2}
	\end{cases}
\end{equation*}

Note that $ V(\widetilde{H}) \cap \{x_{n+1} = 0\} = V(H) $ for any $ H \in \calA $. So the arrangement $ \widetilde{A} = \{\widetilde{H} \mid H \in \calA\} $ is a lifting of $ \calA $. We then verify that $ \widetilde{A} $ meets the requirement. 

By the choice of $ (z_{0},\cdots,z_{n}) $, we have
\begin{equation*}
	\dim  (V(\widetilde{S_{1}}) + V(\widetilde{S_{2}})) - \dim (V(S_{1}) + V(S_{2})) = \begin{cases}
		 1 & V(S_{1}) + V(S_{2}) = \CC^{n+1} \\
		 2 & V(S_{1}) + V(S_{2}) \neq \CC^{n+1}
	\end{cases}
\end{equation*}

Note that
\begin{eqnarray*}
	r(\widetilde{S}) + n+2 - \dim  (V(\widetilde{S_{1}}) + V(\widetilde{S_{2}})) & = & r(\widetilde{S_{1}}) +r(\widetilde{S_{2}}) \\
	& = & r(S_{1}) + r(S_{2}) \\
	& = & r(S) + n+1 - \dim (V(S_{1}) + V(S_{2})).
\end{eqnarray*}

 So we have
\begin{equation*}
	r(\widetilde{S}) = r(S)\iff \CC^{n+1} = V(S_{1}) + V(S_{2}) \iff r(S) = r(S_{1}) + r(S_{2}).
\end{equation*}
\end{proof}

Using the liftings constructed in Lemma \ref{lem: lifting given by partition}, we can prove the following proposition.

\begin{prop}\label{prop: partition criterion for n-vanishing}
		Let $ \calA $ be a hyperplane arrangement in $ \PP^{n} $. Let $ \LL $ be a complex rank-one local system on $ M(\calA) $ with monodromy map $ m \colon \calA \rightarrow \CC^{\times} $. If there exists a bipartition $ \calA = \calA_{1} \sqcup \calA_{2} $ such that
		\begin{enumerate}
			\item $ \prod_{H \in \calA_{1}} m(H) \neq 1 $, and
			\item $ r(S \cap \calA_{1}) + r(S \cap \calA_{2}) = r(S) $ for every subset $ S \subseteq \calA $ of rank $ \leq n $,
		\end{enumerate}
		then $ \calA $ is $ n $-vanishing.
\end{prop}

\begin{proof}
	We prove by induction on the dimension $ n $. When $ n = 1 $, the conclusion holds since $ \LL $ is non-constant. 
	 
	 Suppose the conclusion holds when dimension $ <n $. For the case of dimension $ n $, by Lemma \ref{lem: lifting given by partition} and condition (2) in the proposition, there exists a lifting $ \widetilde{\calA} $ of $ \calA $ such that for any subset $ S \subseteq \calA $,
	\begin{equation*}
		  r(S) = r(S \cap \calA_{1}) + r(S \cap \calA_{2})  \iff r(\widetilde{S}) = r(S) \iff r(\widetilde{S}) \leq n+1.
	\end{equation*}
	
	By condition (1) in the proposition, there exists $ H_{1} \in \calA_{1} $ such that $ m(H_{1}) \neq 1 $. We then claim that for any subset $ S \subseteq \calA $ such that $ H_{1} \in S $ and $ \widetilde{S} \in \RF(\widetilde{\calA},\LL) $,  $ (\widetilde{\calA})^{\widetilde{S}} $ is $ (n+1-r(\widetilde{S})) $-vanishing. 
		
	Since $ r(\widetilde{S}) \leq n+1 $, we have $ r(\widetilde{S}) =  r(S) = r(S \cap \calA_{1}) + r(S \cap \calA_{2})  = r(\widetilde{S} \cap \widetilde{\calA_{1}}) + r(\widetilde{S} \cap \widetilde{\calA_{2}}) $. Note that $ \widetilde{H_{1}} \in \widetilde{S} $ and $ \widetilde{S} $ is irreducible. So $ \widetilde{S} \subseteq \widetilde{\calA_{1}} $, or equivalently, $ S \subseteq \calA_{1} $.

 The claim obviously holds when $ r(\widetilde{S}) = n+1 $. So we may assume that $ r(\widetilde{S}) \leq n $. By definition, 
	\begin{equation*}
		(\widetilde{\calA})^{\widetilde{S}} = \{\widetilde{H} \cap Z(\widetilde{S}) \mid H \in \calA_{1} \setminus S\} \cup  \{\widetilde{H} \cap Z(\widetilde{S}) \mid H \in \calA_{2}\}.
	\end{equation*}
	We then prove this is a disjoint union. It suffices to prove that $ \widetilde{S} \cup \{\widetilde{H_{2}}\} \in L(\widetilde{\calA}) $ for any $ H_{2} \in \calA_{2} $. Otherwise, there exists $ H \in \calA \setminus (S \cup \{H_{2}\}) $ such that $ \widetilde{H} \supseteq \widetilde{H_{2}} \cap Z(\widetilde{S}) $. So $ r(\widetilde{S} \cup \{\widetilde{H},\widetilde{H_{2}}\}) = r(\widetilde{S}) + 1  \leq n+1 $, which implies that
	\begin{equation*}
	 r(\widetilde{S}) + 1 = r(\widetilde{S} \cup \{\widetilde{H},\widetilde{H_{2}}\}) = 	r(\widetilde{S} \cup (\{\widetilde{H}\} \cap \widetilde{\calA_{1}})) + r( \widetilde{H_{2}} \cup (\{\widetilde{H}\} \cap \widetilde{\calA_{2}})).
	\end{equation*}

	However, whenever $ H \in \calA_{1} $ or $ H \in \calA_{2} $, we always have
	\begin{equation*}
		 r(\widetilde{S} \cup (\{\widetilde{H}\} \cap \widetilde{\calA_{1}})) + r( \widetilde{H_{2}} \cup (\{\widetilde{H}\} \cap \widetilde{\calA_{2}})) = r(\widetilde{S})+2,
	\end{equation*}
	a contradiction.
	
	Since $ m(\widetilde{S}) = 1 $, the product of $ m(H') $ over all $ H' \in \{\widetilde{H} \cap Z(\widetilde{S}) \mid H \in \calA_{1} \setminus S \} $ is
	\begin{equation*}
		\prod_{H \in \calA_{1} \setminus S}m(H) = m(\widetilde{S})^{-1} \cdot \prod_{H \in \calA_{1}} m(H) \neq 1. 
	\end{equation*}
	
	 Furthermore, every subset of $ (\widetilde{\calA})^{\widetilde{S}} $ can be written as $ \{\widetilde{H} \cap Z(\widetilde{S}) \mid H \in P \} $ for some $ P \subseteq \calA $ with $ P \cap S = \emptyset $. Its rank is exactly 
	  \begin{equation*}
	  	r(\{\widetilde{H} \cap Z(\widetilde{S}) \mid H \in P \}) = \dim Z(\widetilde{S}) - \dim Z(\widetilde{P}\cup \widetilde{S}) = r(\widetilde{P} \cup \widetilde{S}) - r(\widetilde{S}).
	  \end{equation*}
  
	  So when $ r(\{\widetilde{H} \cap Z(\widetilde{S}) \mid H \in P \}) \leq n+1-r(\widetilde{S}) $, we have $ r(\widetilde{P} \cup \widetilde{S}) \leq n+1 $, and 
	   \begin{equation*}
	  	r(\{\widetilde{H} \cap Z(\widetilde{S}) \mid H \in P \})  = r(P \cup S) - r(S) = r((P \cap \calA_{1})\cup S) + r(P\cap \calA_{2}) - r(S).
	  \end{equation*}
	  In particular,
	  \begin{eqnarray*}
	  		r(\{\widetilde{H} \cap Z(\widetilde{S}) \mid H \in P \cap \calA_{1} \}) + r(\{\widetilde{H} \cap Z(\widetilde{S}) \mid H \in P \cap \calA_{2}\}) &  = & r((P \cap \calA_{1})\cup S) - r(S) + r(P\cap \calA_{2}) \\
	  		& = & r(\{\widetilde{H} \cap Z(\widetilde{S}) \mid H \in P \}).
	  \end{eqnarray*}

	Note that $ r(\widetilde{S}) \geq 2 $ since $ H_{1} \in S $ and $ m(H_{1}) \neq 1 $. So $ n+1-r(\widetilde{S}) \leq n-1 $. By the induction hypothesis, we have $ (\widetilde{\calA})^{\widetilde{S}} $ is $ (n+1-r(\widetilde{S})) $-vanishing. So the claim holds. Then by Corollary \ref{cor: k-vanishing given by R(H)}, $ \widetilde{\calA} $ is $ n $-vanishing. 
	
	Note that for any flat $ \calF \in L(\calA) $ of rank $ \leq n $, we have $ r(\widetilde{\calF}) \leq n+1 $, which implies that $ r(\widetilde{\calF}) = r(\calF) $. So no nonempty flat of $ \calA $ satisfies the condition (\ref{eqn: condition for nonresonant lifting}). Hence by Lemma \ref{lem: dimension shift for lifting}, we have $ \calA $ is $ n $-vanishing. Then the conclusion holds by induction.
\end{proof}

\begin{rmk}
	In checking condition (2) in Proposition \ref{prop: partition criterion for n-vanishing}, it suffices to consider only flats $ \calF \in L(\calA) $ of rank $ \leq n $, rather than all subsets. Moreover, condition (2) holds automatically when $ \calA_{2} $ is in general position with respect to $ \calA_{1} $.
\end{rmk}

For line arrangements, a stronger conclusion holds, namely Theorem \ref{main: 2-vanishing}.

\begin{thm}[Theorem \ref{main: 2-vanishing}]
		Let $ \calA $ be a line arrangement in $ \PP^{2} $ and $ \LL $ a complex rank-one local system on $ M(\calA) $ with monodromy map $ m \colon \calA \rightarrow \CC^{\times} $. If there exists a bipartition $ \calA = \calA_{1} \sqcup \calA_{2} $ such that
	\begin{enumerate}
		\item for $ i=1,2 $, there exists $ H_{i} \in \calA_{i} $ with $ m(H_{i}) \neq 1 $, and
		\item for any resonant flat $ \calF \in \RF(\calA,\LL) $, there exists some $ i $ such that $ \calF \subseteq \calA_{i} $,
	\end{enumerate}
	then the pair $ (\calA,\LL) $ is nonresonant.
\end{thm}

\begin{proof}
	By Lemma \ref{lem: lifting given by partition}, there exists a lifting $ \widetilde{\calA} $ such that for any subset $ S \subseteq \calA $,
	\begin{equation*}
		r(\widetilde{S}) = r(S) \iff r(S) = r(S \cap \calA_{1}) + r(S \cap \calA_{2}).
	\end{equation*}	
	
	Recall that each $ \calA_{i} $ contains a hyperplane $ H_{i} $ such that $ m(H_{i}) \neq 1 $. We then claim that for any subset $ S \subseteq \calA $ such that $ H_{1} \in S $ and $ \widetilde{S} \in \RF(\widetilde{\calA},\LL) $,  $ (\widetilde{\calA})^{\widetilde{S}} $ is $ (3-r(\widetilde{S})) $-vanishing. 
	
	Since $ m(H_{1}) \neq 1 $, we have $ r(\widetilde{S}) \geq r(S) \geq 2 $. The claim obviously holds when $ r(\widetilde{S}) = 3 $. So we may assume that $ r(\widetilde{S}) = r(S) = 2 $. In this case, the claim is equivalent to the statement that there exists $ P \supset S $ such that $ \widetilde{P} \in L(\widetilde{\calA}) $, $ r(\widetilde{P}) = 3 $ and $ m(\widetilde{P}) \neq 1 $. 
	
	Since $ r(\widetilde{S}) = r(S) $, we have $ r(\widetilde{S}) =  r(S) = r(S \cap \calA_{1}) + r(S \cap \calA_{2})  = r(\widetilde{S} \cap \widetilde{\calA_{1}}) + r(\widetilde{S} \cap \widetilde{\calA_{2}}) $. Note that $ \widetilde{H_{1}} \in \widetilde{S} $ and $ \widetilde{S} $ is irreducible. So $ \widetilde{S} \subseteq \widetilde{\calA_{1}} $, or equivalently, $ S \subseteq \calA_{1} $. In particular, $ H_{2} \notin S $.
	
	When $ S \in L(\calA) $, let $ P = S \cup \{H_{2}\} $. Then $ r(P) = 3 = r(S) + r(\{H_{2}\}) $. So $ r(\widetilde{P}) = r(P) = 3 $. Furthermore, if there exists $ H \in \calA \setminus P $ such that $ \widetilde{H} \supseteq Z(\widetilde{P}) $, then $ r(\widetilde{P} \cup \{\widetilde{H}\}) = 3 = r(P \cup \{H\}) $, which implies that
	\begin{equation*}
		r(S \cup (\{H\} \cap \calA_{1})) + r(\{H_{2}\} \cup \{H\} \cap \calA_{2})) = r(P \cup \{H\}) = 3.
	\end{equation*}
	
	However, whenever $ H \in \calA_{1} $ or $ H \in \calA_{2} $, we always have
	\begin{equation*}
		r(S \cup (\{H\} \cap \calA_{1})) + r(\{H_{2}\} \cup \{H\} \cap \calA_{2})) = 4,
	\end{equation*}
	a contradiction. So $ \widetilde{P} \in L(\widetilde{\calA}) $. Moreover, $ m(\widetilde{P}) = m(H_{2}) \neq 1 $.
	
	When $ S \notin L(\calA) $, let $ P = \{H \in \calA \mid Z(S) \subseteq H\} $. Then $ \widetilde{P} \supsetneq \widetilde{S} $ is a flat of $ \widetilde{\calA} $. So $ r(\widetilde{P}) > r(\widetilde{S}) = r(S) = r(P) = 2 $. Hence $ r(\widetilde{P}) = 3 $ and $ r(P) \neq r(P \cap \calA_{1}) + r(P \cap \calA_{2}) $, which implies that $ P $ is irreducible and $ P \notin L(\calA_{1}) \cup L(\calA_{2}) $. Then by condition (2) in the theorem, we have $ m(\widetilde{P}) = m(P) \neq 1 $.

	So the claim holds in both cases. Then by Corollary \ref{cor: k-vanishing given by R(H)}, we have $ \widetilde{\calA} $ is $ 2 $-vanishing. 
	
	Let $ \calF \in L(\calA) $ be a flat of rank $ \leq 2 $ with $ r(\widetilde{\calF}) = r(\calF)+1 $. Then $ r(\calF) \neq r(\calF \cap \calA_{1}) + r(\calF \cap \calA_{2}) $. Hence $ \calF $ is irreducible and $ \calF \notin L(\calA_{1}) \cup L(\calA_{2}) $. By by condition (2) in the theorem, we have $ m(\calF) \neq 1 $. So no nonempty flat of $ \calA $ satisfies the condition (\ref{eqn: condition for nonresonant lifting}). Then by Lemma \ref{lem: dimension shift for lifting}, since $ \widetilde{\calA} $ is $ 2 $-vanishing, we have $ \calA $ is $ 2 $-vanishing.
\end{proof}

Translated into graph-theoretic language, Theorem \ref{main: 2-vanishing} directly implies the following criterion.

\begin{cor}
	Let $ \calA $ be a line arrangement in $ \PP^{2} $. Let $ \LL $ be a complex rank-one local system on the complement $ M(\calA) $ with monodromy map $ m \colon \calA \rightarrow \CC^{\times} $ such that $ m(H) \neq 1 $ for all lines $ H $ in $ \calA $. Let $ \Gamma(\calA,\LL) $ be the graph with vertex set $ \calA $ where two distinct vertices $ H,H' $ are adjacent if and only if $ H \cap H' $ is not a resonant point. If $ \Gamma(\calA,\LL) $ contains a spanning complete bipartite subgraph, then the pair $ (\calA,\LL) $ is nonresonant.
\end{cor}

	\bibliographystyle{plain}
	\bibliography{reference}
\end{document}